\documentclass[a4 paper,11pt]{amsart}

\usepackage{amsmath} %
\usepackage{amssymb} %
\usepackage{amscd} %
\usepackage{amsthm} %
\usepackage{mathrsfs} %
\usepackage{euscript}
\usepackage{eufrak}
\usepackage{ulem}
\usepackage[alphabetic]{amsrefs}
\usepackage{ascmac} %
\usepackage{comment}
\usepackage{bm}%
\usepackage{enumerate}
\usepackage[dvipdfmx]{graphicx} 
\usepackage{hyperref}
\usepackage{xcolor}
\definecolor{unbleu}{rgb}{0.03, 0.15, 0.4}

\hypersetup{
pdfborder = {0 0 0},
colorlinks,
linkcolor=unbleu,
citecolor=unbleu,
urlcolor=unbleu
}

\usepackage{fancyhdr}
 \newtheorem{theorem}{Theorem}[section]
 \newtheorem{lemma}[theorem]{Lemma}
 \newtheorem{proposition}[theorem]{Proposition}
 \newtheorem{problem}[theorem]{Problem}

\newtheorem{corollary}[theorem]{Corollary}

\newtheorem{maintheorem}{Main Theorem} 
  
\theoremstyle{definition}
\newtheorem{definition}[theorem]{Definition}
\newtheorem{remark}[theorem]{Remark}
\newtheorem{example}[theorem]{Example}

\newcommand{\R}{\mathbb R}%     real number
\newcommand{\N}{\mathbb N}%     natural number
\begin{document}

\title[]{Explicit formulae and topological descriptions of action minimizing sets of a full shift with an uncountable alphabet $[0,1]$}

\author[Y.Kajihara]{Yuika Kajihara}
\address{Department of Mathematics, Kyoto University, Kitashirakawa Oiwake-cho, Sakyo-ku,
Kyoto, 606-8502, Japan}
\email{kajihara.yuika.6f@kyoto-u.ac.jp}

\author[S. Motonaga]{Shoya Motonaga}
\address{Faculty of Computer Science and Systems Engineering, Kyushu Institute of Technology, 680-4 Kawazu, Iizuka city, Fukuoka 820-8502, Japan}
\email{motonaga.shoya640@mail.kyutech.jp}

\author[M. Shinoda]{Mao Shinoda}
\address{Department of Mathematics, Ochanomizu University, 2-1-1 Otsuka, Bunkyo-ku, Tokyo, 112-8610, Japan}
\email{shinoda.mao@ocha.ac.jp}

\subjclass[2020]{\textcolor{black}{Primary} 37B10, 37E05, 37E40, 37J51}

\keywords{}

\begin{abstract}
We completely solve ergodic optimization of a full shift with an uncountable alphabet $[0,1]$, which is one of the most well-known examples of infinite dimensional dynamical systems with positive mean dimension (and thus with infinite topological entropy), for potentials depending only on the first two coordinates with the twist condition as well as giving explicit formulae of the associated Mather set and the Aubry set.
Moreover, we investigate the total disconnectedness of the (quotient) Aubry set, in which case the differentiability of the potential function makes a crucial difference. Although these results imply that the (quotient) Aubry set is small enough, we give a complete characterization of an analogical object of the Aubry set, called the Ma\~{n}\'e set, and show that it is much larger than the Aubry set so that it contains cubes of any finite dimension.
 In our proofs, estimates for ``connecting orbits" play key roles, and we establish them by combining three different perspectives in our symbolic setting; weak KAM subaction approach for symbolic dynamics, Ma\~{n}\'e's formulation in Lagrangian systems, and Bangert's variational approach for twist maps.
\end{abstract}

%Lipschitz potentials depending only on the first two coordinates
%potentials depending only on the first two coordinates with the twist condition

\maketitle

\section{Introduction}\label{sec:intro}
\subsection{Backgrounds and our motivations}
Ergodic optimization originates from problems in chaos control studied by Hunt and Ott \cite{HuntOtt1996B, HuntOtt1996A}, and can be regarded as a method to extract ordered dynamics from chaotic systems. Since its formulation by Jenkinson \cite{Jenkinson2006}, it has developed into an active area within ergodic theory
because it naturally appears in the zero-temperature limit of thermodynamic formalism and thus plays an important role in understanding equilibrium states \cite{Contreras2016,Jenkinson19}.
Moreover, in the context of Lagrangian systems, similar problems have been studied in the search for ordered dynamics such as quasi-periodic orbits within non-integrable systems, most notably in Aubry-Mather theory \cite{Mather1982, Mather91}. %These connections reveal a deep interplay between variational principles and dynamical mechanisms.

A central conjecture in ergodic optimization is the Typically Periodic Optimization (TPO) conjecture, which asserts that, for suitable chaotic systems, optimizing measures for generic potentials with suitable regularities are supported on a single periodic orbit.
Various results in this direction have been obtained \cite{YuanHunt1999,Bousch2000, ContrerasLopesThieullen2001,  Contreras2016}. However, for nontrivial concrete potentials of chaotic dynamical systems, explicit descriptions of optimizing measures remain largely unknown.
Here, we shall remark that all optimizing measures of a real analytic potential function of a real analytic expanding circle map have zero entropy unless the potential is cohomologous to a constant \cite{GS2024}, which tells us significant information about the dynamics on the optimal measures for each real analytic potential, but its detailed dynamics are still not well understood.
Even for symbolic dynamical systems with a finite alphabet---which play a fundamental role in the analysis of chaotic systems via Markov partitions---the explicit determination of optimizing measures for concrete potentials has only been partially achieved \cite{ LopesThieullen2003,Jenkinson2006}.

In this paper, we consider a full shift with the uncountable alphabet $[0,1]$, denoted by $\sigma:[0,1]^{\N_0}\to [0,1]^{\N_0}$, where $\mathbb{N}_0$ is the set of non-negative integers.
The phase space $X:=[0,1]^{\N_0}$ is endowed with a metric topology of the metric $d_X$ given by
\[
d_X(\underline{x},\underline{y})=\sum_{i=0}^\infty \frac{|x_i-y_i|}{2^{i+1}},\qquad \underline{x}=x_0x_1x_2\ldots,\  \underline{y}=y_0y_1y_2\ldots \in X, 
\]
and $\sigma:[0,1]^{\N_0}\to [0,1]^{\N_0}$ is a shift map given by 
\[
\sigma(\underline{x})_i=x_{i+1},\qquad \underline{x}=x_0x_1x_2\ldots\in X.
\]
This system is not only obtained by extending symbolic dynamics with finite alphabets to an uncountable one, but rather arises naturally as a fundamental model in geometric analysis, particularly in the theory of mean dimension, since it provides one of the simplest examples of infinite-dimensional dynamical systems with  positive mean dimension \cite{Tsu19} and thus infinite topological entropy.
The ergodic optimization problem for this full shift is also of interest from the point of view of the mean dimension with potentials, especially in the zero-temperature limit \cite{Tsu20,CPV24, Motonaga2025}. Although both ergodic optimization and mean dimension are notoriously difficult to compute in general, our system may offer an ideal model that is tractable.
As stated below, we focus on Lipschitz potentials depending only on the first two coordinates on $X=[0,1]^{\N_0}$, which constitute one of the simplest nontrivial classes of potential functions.

It is worth remarking that the ergodic optimization of the full shift with the uncountable alphabet $[0,1]$ is already studied by Lopes et al. \cite{LMST} from the view point of Markov chains using weak KAM-type techniques such as calibrated subactions (they call this system the \textit{XY model}). They showed generic uniqueness and graph property of optimizing measures (see \cites{BCLMS11,LM14} for details).
However, little is known about the dynamics itself. In particular, neither the TPO property nor the explicit form of optimizing measures for concrete potentials had been clarified.

In our previous work \cite{KMS25}, we introduced analogues of the Aubry set and Ma\~n\'e set for this symbolic dynamics, extending beyond the Mather set, inspired by Ma\~{n}\'e's formulation of action minimizing sets in Lagrangian systems \cite{Sorrentino2015}. Moreover, using weak KAM-type characterizations together with techniques from the analysis of twist maps, we obtained a variational framework sufficient to establish the TPO property for a certain class of potentials
(Lipschitz potentials depending only on the first two coordinates with the twist condition). While periodic optimizing measures for these potentials were completely characterized in \cite{KMS25} (see also Theorem~\ref{thm:criterion} (iii) below), the existence or non-existence of fully supported optimizing measures and other types of measures remained open. Moreover, as in classical Lagrangian systems, Aubry and Ma\~n\'e sets may contain ``connecting orbits" such as homoclinic or heteroclinic ones \cite{CP02, Sorrentino2015}, and it is still unclear what kinds of optimizing orbits can arise.
\begin{problem}
\label{prob:connecting}
What kinds of orbits can exist in the action minimizing invariant sets such as the Mather set, the Aubry set, and Ma\~{n}\'{e} set?
\end{problem}
This situation is not peculiar to our system. To the best of our knowledge, for nontrivial potentials in nontrivial chaotic systems, explicit solutions of ergodic optimization problems as well as explicit optimizing orbits are extremely rare. Clarifying such examples is important not only for ergodic optimization itself but also for understanding the mean dimension with potentials from a concrete viewpoint.

The main contribution of this paper is a complete classification of optimal ``connecting orbits" in the action minimizing set such as the Aubry set and the Ma\~n\'e set for Lipschitz potentials depending only on the first two coordinates with the twist condition, which enables us to solve the ergodic optimization of our systems for these potentials  (Main Theorem~\ref{thm:aubry}).
In general, the Mather set and the Aubry set might be complicated \cite{Mather91} but our results show that, under a twist condition on the potential, these sets become remarkably simple.
On the other hand, the Ma\~n\'e set, which is an extension of the Aubry set, turns out to be significantly larger: it contains cubes of arbitrarily large dimension
(Main Theorem~\ref{thm:semi-static}). This implies that the Ma\~n\'e set contains \textit{meta-chaotic} orbits, which exhibit chaotic behavior for arbitrarily long finite time but eventually reduce to fixed points.

In addition, we analyze the associated equivalence classes (which correspond to action minimizing invariant components) in the Aubry set. Its quotient space is known as the quotient Aubry set, and as in the study of Lagrangian/Hamiltonian dynamics, we address the following natural question:
How small is the quotient Aubry set in some sense?
This type of question dates back to Mather \cite{Mather03} and has a significant importance because, in nearly integrable Hamiltonian systems, a family of invariant tori (corresponding to action minimizing sets) makes a cantor set.
Thus, for more general non-integrable systems, geometric descriptions of action minimizing invariant components correspond to that of the quotient Aubry set.
We will discuss the topological aspect of this question (known as Mather's problem in Lagrangian/Hamiltonian systems) in our setting:
\begin{problem}[An analogy of Mather's problem]\label{prob:Mather}
If the potential has a sufficient regularity, is the quotient Aubry set totally disconnected?
\end{problem}
For the setting of Lagrangian systems, there are some relevant results.
In the original paper \cite{Mather03}, Mather provided a sufficient condition for the total disconnectedness of the quotient Aubry set.
Sorrentino \cite{Sorrentino08} generalized this result for higher dimensional systems, assuming sufficient regularity and natural conditions (in terms of applications) of Lagrangians (which correspond to potential functions in our setting).
The assumption of regularity of Lagrangians is crucial for the total disconnectedness of the quotient Aubry set since, without sufficient regularity, there are examples such that the quotient Aubry sets are isometric to closed intervals \cite{Mather04}.
See also \cite{FFR09} for the related research on Hausdorff dimension.
We show similar (perhaps more detailed) results for topological descriptions of the  Aubry set and its quotient space for our system (Main Theorem~\ref{thm:Aubry_small}). In particular, we give an affirmative answer to Mather's problem in our setting as well as providing a sufficient condition that the quotient Aubry sets are isometric to the unit interval $[0,1]$.

Our approach in this paper relies on three different methods:
weak KAM approach, Ma\~{n}\'e's formulation in Lagrangian systems, and Bangert's variational approach for twist maps.
It is known that these methods are closely related with each other in Lagrangian/Hamiltonian dynamical systems \cite{Sorrentino2015}, but there is no trivial connection in our symbolic setting.
We focus on ``connecting orbits" as well as equivalent classes in the Aubry set, and investigate them from these three perspectives.
Most of computations and discussion in each step of the proofs are entirely elementary, but switching the above different perspectives and results seems not obvious.
In the next subsection, we will describe our precise settings and statements.

\subsection{Main results}
{Now we consider the ergodic optimization of the uncountable full shift $([0,1]^{\N_0},\sigma)$ for Lipschitz continuous functions on $[0,1]^{\N_0}$, especially functions depending only on the first two coordinates with the twist condition.}
A $\sigma$-invariant Borel probability measure that attains the \textit{optimal ergodic average}
\[
\alpha_{\varphi}=\inf_{\mu\in \mathcal{M}_\sigma(X)} \int \varphi d\mu
\]
is called an {\it minimizing measure} for $\varphi\in C(X)$ (called \textit{potential}), where $C(X)$ is the set of continuous functions on $X$ and $\mathcal{M}_\sigma(X)$ stands for the set of $\sigma$-invariant Borel probability measures on $X$ endowed with the weak$^*$-topology.
The set of minimizing measures for $\varphi$ is denoted by $\mathcal{M}_{{\rm min}}(\varphi)$.
One of {the aim of this paper is to understand the complete description of  $\mathcal{M}_{{\rm min}}(\varphi)$ as well as} the ``complexity" of the corresponding invariant set, called the \textit{Mather set} of $\varphi$, defined as
\[
\mathscr{M}_\varphi=\bigcup_{\mu\in \mathcal{M}_{{\rm min}}(\varphi)}{\rm supp}(\mu),
\]
where ${\rm supp}(\mu)$ is the intersection of all compact sets with full measure with respect to $\mu$.

A common approach to investigate the Mather set is considering calibrated subactions (see Section~\ref{sec:calibrated_equivalent} for the details).
In addition, it is useful to consider a larger invariant set, called the Aubry set, for a given Lipschitz potential.
In \cite{BLL13}, inspired by the weak KAM theory, detailed descriptions of the Mather set and the Aubry set as well as calibrated subactions for a subshift of finite type are discussed, {but the complete dynamics on the support of minimizing measures for concrete potentials are still not clear}.
Using their ideas, {as stated before,} we also obtained analogous results of the Mather sets and the Aubry sets for Lipschitz functions with respect to the full shift with $[0,1]$ in the previous work \cite{KMS25}.
Furthermore, combining variational techniques developed in \cites{Ban88, Yu22}, we gave new descriptions of the Aubry sets for Lipschitz potentials depending only on the first two coordinates with the twist condition.

{Below we summarize our previous results since the present work combines these results with the analysis of ``connecting orbits":}
Let $\Omega_\varphi$ and $N_\varphi$ be the Aubry set and the Ma\~{n}\'{e} set for a Lipschitz function $\varphi$, and let $H_\varphi:X\times X\to \R\cup\{+\infty\}$ be the Peierl's barrier (see Section~\ref{sec:action_minimizing} for their precise definitions).
Then, results similar to Aubry-Mather theory for Euler-Lagrange flows are also valid in our setting:
\begin{theorem}[\cite{KMS25}]
\label{theorem:KMS25}
Let $\varphi$ be a Lipschitz continuous function on $X=[0,1]^{\N_0}$. Then
\begin{enumerate}[(i)]
    \item $\mathscr{M}_\varphi\subset \Omega_\varphi=\{\underline{x}\in X \mid H_\varphi(\underline{x},\underline{x})=0\}\subset N_\varphi$.
    \item If both \( \underline{x} \) and \( \underline{y} \) belong to \( \Omega_{\varphi} \), then the relation \( \underline{x} \sim_\varphi \underline{y} \) given by  
    \[
    H_{\varphi}(\underline{x},\underline{y})+H_{\varphi}(\underline{y},\underline{x})=0
    \]
    is an equivalence relation on $\Omega_{\varphi}$.
\end{enumerate}
\end{theorem}
Restricting our attentions to Lipschitz potentials depending only on the first two coordinates with the twist condition, we have much more explicit information on the Aubry set.
For a continuous function $h:[0,1]^2\to\R$, we introduce the following notations, which we frequently use in the present paper:
\[
h^\ast:=\min_{x \in [0,1]} h(x,x),\qquad 
 \mathrm{m}_h:=\{a\in[0,1]\mid h(a,a)= h^\ast\}.
\]
Moreover, when $h:[0,1]^2\to\R$ is of class $C^2$, we denote its derivative with respect to the $i$-th component by $D_i$ for $i=1,2$ respectively.
Let $\mathscr{H}$ be the set of $C^2$-functions with the twist condition on $[0,1]^2$, that is, 
\[
\mathscr{H}=\{h\in C^2([0,1]^2;\R)\mid D_2 D_1 h<0 \}.
\]
By abuse of notation, we write ``$\varphi \in \mathscr{H}$" for a Lipschitz continuous function $\varphi$ on $X$ if there exists $h\in \mathscr{H}$ satisfying $\varphi(\underline{x})=h(x_0,x_1)$ for all $\underline{x}=x_0x_1x_2\ldots\in X$.

\begin{theorem}[Main Theorem 3 in \cite{KMS25}]
\label{thm:criterion}
    Suppose that $\varphi(\underline{x})=h(x_0,x_1)$ with some $h\in \mathscr{H}$.
    Then we have $(i)$--$(iii)$:
    \begin{enumerate}[(i)]
    \item $\alpha_\varphi= h^\ast$.
    \item $\mathcal{M}_{{\rm min}}(\varphi)\cap \mathcal{M}^{\mathrm{p}}=\{\delta_{a^\infty}\mid a\in \mathrm{m}_h\}$,
    where $\delta_{\underline{x}}$ is the Dirac measure supported at $\underline{x}$ and $\mathcal{M}^{\mathrm{p}}$ stands for the set of invariant probability measures supported on a single periodic orbit.
    \item $\Omega_\varphi\subset (\mathrm{m}_h)^{\N_0}$, i.e., for any $\underline{x} = \{x_i\}_{i \in \N_0} \in \Omega_{\varphi}$, it holds that
    \[h(x_i,x_i)= h^\ast \ \text{for all} \ i \in \N_0.\]
    \end{enumerate}
\end{theorem}
Note that for generic $h\in\mathscr{H}$, $\mathrm{m}_h$ consists of a single point, and then the associated Mather set and Aubry set become a single fixed point, which leads to the typically periodic optimization in the class of Lipschitz potentials depending only on the first two coordinates with the twist condition for the full shift with $[0,1]$ (see Main Theorem~4 in \cite{KMS25}).

{In the present paper, apart from typical properties of minimizing measures, we focus on explicit minimizing measures for each potential in $\mathscr{H}$ (or a larger functional space).}
Actually we investigate all the possibilities of minimizing orbits for such potentials, and it enables us to obtain the ones of minimizing measures.
Our first main result provides complete characterizations of the Mather set and the Aubry set of $\varphi \in \mathscr{H}$, which is an improvement of \cite{KMS25} and tells us that these action minimizing invariant sets $ \mathscr{M}_{\varphi}$ and $\Omega_{\varphi}$ are much simple.
\begin{maintheorem}
\label{thm:aubry}
Suppose that $\varphi(\underline{x})=h(x_0,x_1)$ with some $h\in \mathscr{H}$. Then
\[
\mathscr{M}_\varphi=\Omega_{\varphi}=\bigcup_{a \in \mathrm{m}_h} \{a^\infty\}.
\]
{
In particular,
\[
\mathcal{M}_{{\rm min}}(\varphi)=\{\delta_{a^\infty}\mid a\in \mathrm{m}_h\}.
\]
}
\end{maintheorem}

Next, we want to discuss the ``smallness" of the associated action minimizing sets.
In order to discuss the essential components of the Aubry set, we introduce the quotient Aubry set, which is originally considered in the Aubry-Mather theory for Euler-Lagrange flows \cites{Mather91, Mather03}.
Since
\[	\delta_\varphi(\underline{x},\underline{y}):=H_\varphi(\underline{x},\underline{y})+H_\varphi(\underline{y},\underline{x})
\]
is a pseudo-metric on $\Omega_\varphi$ (see Section~\ref{sec:equivalence-Aubry}),
we can define an equivalence relation $\underline{x}\sim_\varphi\underline{y}$ defined by $\delta_\varphi(\underline{x},\underline{y})=0$
(see Theorem \ref{theorem:KMS25}(ii)).
Then $\delta_\varphi$ induces a metric on $\bar{\Omega}_\varphi:=\Omega_\varphi/\sim_\varphi$.
We call $\bar{\Omega}_\varphi$ the \textit{quotient Aubry set}, as in the Aubry-Mather theory for Euler-Lagrange flows.

Before stating our second main theorem, let us define another equivalence relation.  
For a fixed $h\in\mathscr{H}$, we write  
\[
a \sim_{\mathrm{conn},h} b
\]  
for $a,b\in\mathrm{m}_h$ if $a$ and $b$ belong to the same connected component (denoted by $C(a)$) of $\mathrm{m}_h \subset \mathbb{R}$.
Introducing a pseudo-metric on $\mathrm{m}_h$ (see Proposition~\ref{prop:pseudo-metric} for its details) defined by
\[
\hat{d}_{\R}(a,b)=d_H(C(a),C(b)),
\]
where $d_H$ stands for the Hausdorff distance, we see that $a\sim_{\mathrm{conn},h} b$ holds if and only if $\hat{d}_\R(a,b)=0$.
Thus, the pseudo-metric $\hat{d}_{\R}$ on ${\mathrm{m}}_h$ induces the metric of the quotient space $\bar{\mathrm{m}}_h:=\mathrm{m}_h/\sim_{\mathrm{conn},h}$ and it is a totally disconnected space.
The next result describes the ``size'' of the quotient Aubry sets for our symbolic dynamics.
\begin{maintheorem}\label{thm:Aubry_small}
Suppose that $\varphi(\underline{x})=h(x_0,x_1)$ with some Lipschitz continuous function $h:[0,1]^2\to\R$.
Then the following statements hold:
\begin{enumerate}[(i)]
    \item If $h\in\mathscr{H}$, the Aubry set $\Omega_{\varphi}$ is isometric to $\mathrm{m}_h\subset\R$.
    \item If $h\in\mathscr{H}$, the quotient Aubry set $\bar{\Omega}_\varphi$ is homeomorphic to the totally disconnected space $\bar{\mathrm{m}}_h$.
    \item If $h$ is of the form $h=\rho(x-y)+\frac{1}{2} |x-y|$ with a $C^2$-function $\rho:\R\to\R$ such that $\rho''>0$ on $\R$,
    then the quotient Aubry set $(\bar{\Omega}_{\varphi},\delta_{\varphi})$
    is isometric to the unit interval $([0,1],d_\R)$.
\end{enumerate}
\end{maintheorem}

Main Theorem~\ref{thm:Aubry_small}(i) (resp. (ii)) implies that the Aubry set (resp. the quotient Aubry set) for $\varphi\in\mathscr { H}$ is very small.
Since the Aubry set $\Omega_\varphi$ is a subset of the Ma\~{n}\'{e} set $N_\varphi$, it is natural to ask whether the Ma\~{n}\'{e} set $N_\varphi$ is also small or not.
We have completely determined $N_\varphi$ by the following result, which implies that it is much larger than the Aubry set $\Omega_\varphi$.
\begin{maintheorem}\label{thm:semi-static}
	Let $\varphi:X\to\R$ be a Lipschitz continuous function. Then the Ma\~{n}\'{e} set $N_{\varphi}$ of $\varphi$ can be expressed as
    \[
N_{\varphi}=\{ \underline{x} \in X \mid \sigma^{i}(\underline{x}) \neq \sigma^{j}(\underline{x}) \ \text{if} \ i \neq j \}\cup \bigcup_{M\in\N_0} \sigma^{-M}(\mathscr{M}_{\varphi}^{\mathrm{p}}),
\]
    where $\mathscr{M}_\varphi^{\mathrm{p}}$ stands for the set of periodic points contained in the Mather set $\mathscr{M}_\varphi$ of $\varphi$.

    In particular, for a Lipschitz potential depending only on the first two coordinates  $\varphi(\underline{x})=h(x_0,x_1)$ with some $h\in \mathscr{H}$, it holds that
\[
N_{\varphi}=\{ \underline{x} \in X \mid \sigma^{i}(\underline{x}) \neq \sigma^{j}(\underline{x}) \ \text{if} \ i \neq j \}\cup \bigcup_{M\in\N_0} \sigma^{-M}(\{a^\infty\mid a\in\mathrm{m}_h\}).
\]
\end{maintheorem}
Note that the above formula implies that $N_{\varphi}$ contains cubes of any finite dimension.

The rest of this paper is organized as follows.
In Section \ref{sec:pre_action-minimizing-sets}, we review the basic definitions and properties of the action minimizing sets, and investigate some properties of calibrated subactions in terms of their connection with the equivalence relation defined on the Aubry set.
Section \ref{sec:explicit-form} discusses explicit formulae for the action minimizing sets for $\varphi\in\mathscr{H}$.
In Section \ref{sec:equivalence-Aubry}, we study the equivalence classes of the Aubry set and provide topological descriptions of the (quotient) Aubry set.

\section{Preliminaries: action minimizing sets and calibrated subactions}
\label{sec:pre_action-minimizing-sets}
Throughout this section, we only assume that $\varphi$ is Lipschitz continuous.
\subsection{Action minimizing sets for the full shift with $[0,1]$}\label{sec:action_minimizing}
In this subsection, we briefly review the definitions and basic properties of the Aubry set and the Ma\~{n}\'e set.
The definition of the Aubry set is based on the {\it {Ma\~{n}\'e potential}}.

\begin{definition}[Ma\~n\'e potential]
\label{defi:mane-potential}
    For a Lipschitz function
    $\varphi:X\rightarrow\mathbb{R}$ and $\varepsilon>0$ define
    \begin{align*}
        \widehat{S}_\varphi(\underline{x},\underline{y};\varepsilon)
        =\inf \left\{
        \sum_{i=0}^{n-1} \left(\varphi\circ\sigma^i(\underline{z})-\alpha_{\varphi}\right) \mid n\in \mathbb{N}, \underline{z}\in B(\underline{x},\underline{y},n;\varepsilon)
        \right\},
    \end{align*}
    where 
    \begin{align*}
        B(\underline{x},\underline{y},n;\varepsilon)=
       \{\underline{z} \in X \mid d_X(\underline{x},\underline{z})<\varepsilon, \ d_X(\sigma^n(\underline{z}),\underline{y})<\varepsilon\}.
    \end{align*}
    Then we introduce the \textit{Ma\~n\'e potential}
    $S_{\varphi}$ given by
    \begin{align*}
        S_\varphi(\underline{x},\underline{y})=\lim_{\varepsilon\to 0 }\widehat{S}_\varphi(\underline{x},\underline{y};\varepsilon),
    \end{align*}
\end{definition}
The Ma\~n\'e potential originates from the context of the Aubry-Mather theory for Euler-Lagrange flows.

\begin{definition}[Aubry set]\label{aubry_set}
    The set $\Omega_\varphi=\{\underline{x}\in X \mid S_\varphi(\underline{x},\underline{x})=0\}$
    is called the \textit{Aubry set} of $\varphi$.
\end{definition}

Note that $\Omega_\varphi$ for a Lipschitz function $\varphi$ is non-empty, $\sigma$-invariant, and compact (see the remark after Proposition~\ref{prop:action_minimizing_aubry}).
Next, let us define Peierl's barrier.
\begin{definition}[Peierl's barrier]
    For a Lipschitz function $\varphi:X\rightarrow \mathbb{R}$ and $\varepsilon>0$, define the \textit{Peierl's barrier} $H_\varphi:X\times X\rightarrow \mathbb{R}\cup\{\infty\}$ as
    \begin{align*}
        H_\varphi(\underline
        {x},\underline{y})=\lim_{\varepsilon\to 0} \widehat{H}_\varphi(\underline{x},\underline{y};\varepsilon),
    \end{align*}
    where
    \begin{align*}
         \widehat{H}_\varphi(\underline{x}, \underline{y};\varepsilon)=\liminf_{n\to \infty}\left\{\sum_{i=0}^{n-1} \left(\varphi\circ\sigma^i(\underline{z})-\alpha_{\varphi}\right) \mid \underline{z}\in B(\underline{x},\underline{y},n;\varepsilon)\right\}.
    \end{align*}
\end{definition}

\begin{remark}
We give comments on $S_{\varphi}$ and $H_{\varphi}$.
\begin{enumerate}[(i)]
    \item Note that $S_\varphi$ does not take $-\infty$ if $\varphi$ is Lipschitz.
    Actually, for a Lipschitz function $\varphi: X\rightarrow \mathbb{R}$, we have
    \begin{align}
        S_\varphi(\underline{x},\underline{y})\geq u(\underline{y})-u(\underline{x})
        \label{eq:mane_calibrated}
    \end{align}
    for a Lipschitz subaction $u$ of $\varphi$ and $\underline{x},\underline{y}\in X$ (see Lemma 2.6 in \cite{KMS25}).
    \item The Peierl's barrier defined as above may take $\infty$. One can see that it holds that $H_{\varphi}(1^\infty, 1^\infty)=\infty$ for $\varphi(\underline{x})=x_0$ (see Remark 2.10 in \cite{KMS25}).
\end{enumerate}
\end{remark}
The Peierl's barrier has the following upper bound.
\begin{theorem}[Theorem 2.11 in \cite{KMS25}]
\label{theorem:H_finite}
    Let $\varphi:X\rightarrow \mathbb{R}$ be a Lipschitz function with a Lipschitz constant $L_\varphi>0$.
    For $\underline{x}\in \Omega_\varphi$ and $\underline{y}\in X$, we have 
    $H_\varphi(\underline{x},\underline{y})\leq L_\varphi d_X(\underline{x},\underline{y})$.
    In particular, it holds that $\underline{x}\in \Omega_\varphi$ if and only if 
    $H_\varphi(\underline{x},\underline{x})=0$.
\end{theorem}

Now, we describe another point of view for the Aubry set.
A positive-semi orbit $\{\sigma^n(\underline{x})\}_{n\in\N_0}$ is said to be
\begin{itemize}
\item $\varphi$-semi-static: if for any non-negative integers $i<j$
\begin{align}\label{eqn:semi_static}
	\sum_{n=i}^{j-1} \left(\varphi\circ\sigma^n(\underline{x})-\alpha_{\varphi}\right)=S_\varphi(\sigma^{i}(\underline{x}),\sigma^j(\underline{x})).
\end{align}
\item $\varphi$-static: if for any non-negative integers $i<j$
\begin{align*}
	\sum_{n=i}^{j-1} \left(\varphi\circ\sigma^n(\underline{x})-\alpha_{\varphi}\right)=-S_\varphi(\sigma^{j}(\underline{x}),\sigma^i(\underline{x})).
\end{align*}
\end{itemize}
Set
\begin{align*}
	N_{\varphi}&:=\{\sigma^k(\underline{x})\in X \mid \{\sigma^n(\underline{x})\}_{n\in\N_0}\ \text{is}\ \varphi\text{-semi-static}, k\in\N_0\},\\
	A_{\varphi}&:=\{\sigma^k(\underline{x})\in X \mid \{\sigma^n(\underline{x})\}_{n\in\N_0}\ \text{is}\ \varphi\text{-static}, k\in\N_0\}.
\end{align*}
These sets are derived from Ma\~{n}\'{e}'s work for Euler-Lagrange flows. Following the context of the Aubry-Mather theory for Euler-Lagrange flows, we call $N_\varphi$ the \textit{Ma\~{n}\'{e} set} for $\varphi$.
From the following proposition, $A_\varphi$ can be referred as the Aubry set for $\varphi$, while $N_\varphi$ and $A_\varphi$ appear in \cite{KMS25} as the \textit{$\varphi$-semi-static set} and \textit{$\varphi$-static set} for $\varphi$ respectively.
\begin{proposition}[Theorem 3.5 in \cite{KMS25}]
\label{prop:action_minimizing_aubry}
For each Lipschitz function $\varphi \colon X \to \R$, we have
    \[
    A_{\varphi}= \Omega_{\varphi}.
    \]
\end{proposition}
The relation among the Mather set $\mathscr{M}_\varphi$, the Aubry set $\Omega_\varphi(=A_\varphi)$ and the Ma\~{n}\'{e} set $N_\varphi$ is summarized in Theorem~\ref{theorem:KMS25}(i).
By Theorem~\ref{theorem:KMS25}(i) and the definitions of $A_\varphi$ and $N_\varphi$, both the Aubry set and the Ma\~{n}\'{e} set for a Lipschitz continuous function $\varphi$ on $X$ are non-empty and $\sigma$-invariant.
In the case of Euler-Lagrange flows, the Aubry set and the Ma\~{n}\'{e} set associated with a so-called Tonelli Lagrangian are compact invariant sets for their associated flows. However, in our setting, compactness (more precisely, closedness) holds only for the Aubry set $\Omega_\varphi$, and the Ma\~{n}\'{e} set $N_\varphi$ may not be closed since $S_\varphi$ is only lower semicontinuous on $X\times X$ (Proposition 2.7 in \cite{KMS25}).
Actually, as a byproduct of Main Theorem ~\ref{thm:semi-static}, we will see that $N_\varphi$ cannot be closed in some cases (see Example~\ref{ex:nonclosed}).

We now give a connection between these invariant sets $A_\varphi,N_\varphi$ and subactions.
\begin{definition}[Subaction]
\label{calibrated_subaction}
    A continuous function $u:X\rightarrow \mathbb{R}$ is called a {\it subaction} of $\varphi$ if
    \begin{align*}
        u(\underline{x})+\varphi(\underline{x})\geq u(\sigma \underline{x})+\alpha_\varphi
    \end{align*}
    for every $\underline{x}\in X$.
\end{definition}

\begin{lemma}[Proposition 3.6 in \cite{KMS25}]
\label{lemm:subaction_static}
Let $u$ be a Lipschitz subaction of a Lipschitz function $\varphi:X\to\R$.
If $\underline{x}\in A_\varphi$, then
    \begin{align}
    \label{each_static}
    u(\sigma^{k+1}(\underline{x}))-u(\sigma^{k}(\underline{x}))=\varphi(\sigma^k(\underline{x}))-\alpha_{\varphi}
    \end{align}
    for all $k\in\N_0$.
    Conversely, if \eqref{each_static} holds for each $k\in\N_0$, then $\underline{x}\in N_\varphi$.
\end{lemma}

From Proposition~\ref{prop:action_minimizing_aubry} and Lemma \ref{lemm:subaction_static}, we immediately obtain the following.
\begin{proposition}\label{bounded}
    If $\underline{x}\in \Omega_\varphi$, then the sequence
    $\displaystyle\left\{\sum_{i=0}^{n-1}\big(\varphi\circ\sigma^i(\underline{x})-\alpha(\varphi)\big)\right\}_{n\in\N_0}$
    is bounded.
\end{proposition}
\begin{proof}
    Let $\underline{x}\in\Omega_\varphi$.
    By Proposition~\ref{prop:action_minimizing_aubry}, we have $\underline{x}\in A_\varphi$.
    By taking the subaction $u$ in Lemma~\ref{lemm:subaction_static},
    we have
    \begin{align*}
    \sum_{k=0}^{n-1}(\varphi(\sigma^k(\underline{x}))-\alpha_{\varphi})=
    \sum_{k=0}^{n-1} \big(u(\sigma^{k+1}(\underline{x}))-u(\sigma^{k}(\underline{x}))\big)=
        u(\sigma^{n}(\underline{x}))-u(\underline{x}),
    \end{align*}
    which implies that the sequence
    \[\left\{\sum_{i=0}^{n-1}\big(\varphi\circ\sigma^i(\underline{x})-\alpha_\varphi \big)\right\}_{n\in\N_0}=
    \left\{u(\sigma^{n}(\underline{x}))-u(\underline{x})\right\}_{n\in\N_0}\]
    is bounded
    since $u:[0,1]^{\N_0}\to\R$ is continuous.
\end{proof}

We next describe several new relations between $\mathscr{M}_\varphi, \Omega_\varphi$ and $N_\varphi$ for Lipscitz continuous functions $\varphi$ on $X$ in the setting of the full shift with $[0,1]$, which is not discussed in \cite{KMS25}.
\begin{proposition}\label{prop:action_minimizing_periodic} For a Lipschitz continuous function $\varphi$ on $X$, it holds that
\[
\mathscr{M}_{\varphi}^{\mathrm{p}}
=\Omega_{\varphi}^{\mathrm{p}}=N_{\varphi}^{\mathrm{p}}
\]
where $\mathscr{M}_{\varphi}^{\mathrm{p}}, \Omega_{\varphi}^{\mathrm{p}}$ and $N_{\varphi}^{\mathrm{p}}$ represent the sets of periodic points in $\mathscr{M}_{\varphi}, \Omega_{\varphi}$ and $N_\varphi$ respectively.
\end{proposition}
\begin{proof}
    From Theorem~\ref{theorem:KMS25}, we immediately have $\mathscr{M}_{\varphi}^{\mathrm{p}}
\subset \Omega_{\varphi}^{\mathrm{p}}\subset N_{\varphi}^{\mathrm{p}}$.
    Now we show that $N_{\varphi}^{\mathrm{p}}\subset \mathscr{M}_{\varphi}^{\mathrm{p}}$.
    Take $\underline{y} \in N_\varphi^\mathrm{p}$ and let $\tau\in\N$ be the minimal period of $\underline{y}$.
    Then the definition of $N_\varphi$ implies that
    \begin{align*}
    S_\varphi(\underline{y},\sigma^{\tau}(\underline{y}))=\sum_{i=0}^{\tau-1} (\varphi\circ\sigma^i(\underline{y})-\alpha_{\varphi}),\quad
S_\varphi(\underline{y},\sigma^{2\tau}(\underline{y}))=\sum_{i=0}^{2\tau-1} (\varphi\circ\sigma^i(\underline{y})-\alpha_{\varphi}),
    \end{align*}
    and the periodicity of $\underline{y}$ yields that
    $S_\varphi(\underline{y},\sigma^{\tau}(\underline{y}))=S_\varphi(\underline{y},\sigma^{2\tau}(\underline{y}))$ and thus
    $\sum_{i=0}^{\tau-1} (\varphi\circ\sigma^i(\underline{y})-\alpha_{\varphi})=0$.
    Therefore, the periodic point $\underline{y}$ satisfies
    \[
    \frac{1}{\tau}\sum_{i=0}^{\tau-1} \varphi\circ\sigma^i(\underline{y})=\alpha_\varphi,
    \]
    which implies that the invariant probability measure evenly distributed on the periodic orbit of $\underline{y}$ is an minimizing measure of $\varphi$, i.e., $\underline{y}\in \mathscr{M}_\varphi^{\mathrm{p}}$.
\end{proof}

In the setting of the full shift with $[0,1]$, we can obtain the following explicit relation of the Ma\~{n}\'{e} set and the Mather set (and the Aubry set), which tells us that $N_\varphi$ is much large so that it is dense in the whole space $X$ (see also Example~\ref{ex:nonclosed}) {and contains cubes of any finite dimension.}

\begin{theorem}[a part of Main Theorem~\ref{thm:semi-static}]
\label{thm:mane-set}
Let $\varphi$ be a Lipschitz continuous function on $X$.
Then it holds that
\[
N_{\varphi}=N_1 \sqcup N_2
\]
where
\begin{align*}
    N_1&=\{ \underline{x} \in X \mid \sigma^{i}(\underline{x}) \neq \sigma^{j}(\underline{x}) \ \text{if} \ i \neq j \}, \ \text{and} \\
    N_2
    &=\bigcup_{M\in\N_0}\sigma^{-M}(\mathscr{M}_{\varphi}^{\mathrm{p}})
    =  \bigcup_{M\in\N_0}\sigma^{-M}(\Omega_{\varphi}^{\mathrm{p}})
\end{align*}
\end{theorem}
We remark that $N_2$ (resp. $N_1$) corresponds to the eventually periodic (resp. non eventually periodic) part of the Ma\~{n}\'{e} set $N_\varphi$ of $\varphi$.

Before the proof of Theorem~\ref{thm:mane-set}, we need the following lemma, which implies that the minimizing ``connecting orbit" from $\underline{x}$ to $\sigma^k(\underline{x})$ is realized by the finite segment of $\underline{x}$ itself without repetition.
\begin{lemma}
\label{lemm:finite-sum}
    For any $\underline{x}\in X$ and $k\in \mathbb{N}$, $S_\varphi(\underline{x}, \sigma^k(\underline{x}))$ attains a finite sum, i.e.,
    \begin{align}\label{eqn:finite_minimal}
    S_\varphi(\underline{x}, \sigma^k(\underline{x}))=
    \sum_{i=0}^{l-1} (\varphi\circ\sigma^i(\underline{x})-\alpha_{\varphi}), \  
    \end{align}
    where
    \[
    l = \min \{m \in \N \mid \sigma^{m}(\underline{x}) = \sigma^k(\underline{x}) \}.
    \]
\end{lemma}
\begin{proof}
    It is easy to see that $\sigma^{l}(\underline{x}) = \sigma^k(\underline{x})$ and
    \[
    S_\varphi(\underline{x},\sigma^k(\underline{x}))
    =S_\varphi(\underline{x},\sigma^l(\underline{x}))\le \sum_{i=0}^{l-1} (\varphi\circ\sigma^i(\underline{x})-\alpha_{\varphi}).
    \]
    We now prove the opposite inequality.
    
    Take $\{n_j\}$ (not necessarily unbounded) and ${\underline{z}^{(j)}}\in B(\underline{x},\sigma^k(\underline{x}),n_j;2^{-j})$ such that
    \[
      \lim_{j\to\infty} {\sum_{i=0}^{n_j-1} (\varphi\circ\sigma^i(\underline{z}^{(j)})-\alpha_{\varphi})}=S_\varphi(\underline{x},\sigma^{k}(\underline{x})).
    \]
     We first show that {$n_j\ge l$} for sufficiently large $j$.
     Let
     \[
     \delta:=\min_{i=1,\ldots, l-1} d(\sigma^l(\underline{x}),\sigma^i(\underline{x})),
     \]
     which is positive by the choice of $l$.
     Take sufficiently large $j\in\N$ so that $2^{-j+l}<\delta$.
     Assume that {$n_j< l$}.
     From the conditions
     $d(\underline{z}^{(j)},\underline{x})<2^{-j}$
     and $d(\sigma^{n_j}(\underline{z}^{(j)}),\sigma^k(\underline{x}))<2^{-j}$, we compute
     \begin{align*}
     \delta
     &\le d(\sigma^l(\underline{x}),\sigma^{n_j}(\underline{x}))\\
     &\le
     d(\sigma^l(\underline{x}),\sigma^{n_j}(\underline{z}^{(j)}))+d(\sigma^{n_j}(\underline{z}^{(j)}),\sigma^{n_j}(\underline{x}))\\
     &< 2^{-j}+2^{-j+n_j}\le  2^{-j+l}<\delta,
     \end{align*}
which yields a contradiction.

Now we fix sufficiently large $j\in\N$ so that {$n_j\ge l$}.
If {$n_j>l$}, we have $\sigma^{l}(\underline{z}^{(j)})\in B(\sigma^l(\underline{x}),\sigma^k(\underline{x}),n_j-l;2^{-j+l})=B(\sigma^l(\underline{x}),\sigma^l(\underline{x}),n_j-l;2^{-j+l})$, and thus it holds that
\begin{align*}
S_\varphi&(\underline{x},\sigma^k(\underline{x}))\\
    &=\lim_{j\to\infty} \left({\sum_{i=0}^{l-1} (\varphi\circ\sigma^i(\underline{z}^{(j)})-\alpha_{\varphi})}+{\sum_{i=l}^{n_j-1} (\varphi\circ\sigma^i(\underline{z}^{(j)})-\alpha_{\varphi})}\right)\\
    &\ge \lim_{j\to\infty} \left({\sum_{i=0}^{l-1} (\varphi\circ\sigma^i(\underline{x})-\alpha_{\varphi})}-\kappa {L_\varphi}2^{-j+l}+S_{\varphi}(\sigma^l(\underline{x}),\sigma^l(\underline{x});2^{-j+l})\right)\\
    &= {\sum_{i=0}^{l-1} (\varphi\circ\sigma^i(\underline{x})-\alpha_{\varphi})}+S_{\varphi}(\sigma^l(\underline{x}),\sigma^l(\underline{x}))\\
    &\ge {\sum_{i=0}^{l-1} (\varphi\circ\sigma^i(\underline{x})-\alpha_{\varphi})}
\end{align*}
for some constant $\kappa>0$, which only depends on $l$.
Here we use \eqref{eq:mane_calibrated} in the last inequality.
The second term of the first line of the above inequalities does not appear when $n_j=l$, and we immediately obtain the claim.
\end{proof}

Now we are ready to prove Theorem \ref{thm:mane-set}.
\begin{proof}[{Proof of Theorem \ref{thm:mane-set}}]
    $N_1\subset N_\varphi$ immediately follows from the definition of $N_1$ and Lemma~\ref{lemm:finite-sum}.
    We now prove $N_2\subset N_\varphi$ by induction.
    By Theorem~\ref{theorem:KMS25}, we have
    $\mathscr{M}_\varphi^{\mathrm{p}}\subset N_\varphi$.
    Assume that $\sigma^{-M}(\mathscr{M}_\varphi^{\mathrm{p}})\subset N_\varphi$ for some $M\in\N_0$. Consider $\underline{x}=b\underline{x}'$ with any $b\in[0,1]$ and $x'\in\sigma^{-M}(\mathscr{M}_\varphi^{\mathrm{p}})$.
    Since $\sigma^{-M}(\mathscr{M}_\varphi^{\mathrm{p}})\subset N_\varphi$,
    we see that \eqref{eqn:semi_static} holds for any $i>0$ and $j>i$.
    Note that, in particular, it holds that
    \begin{align}\label{eqn:Mane_induction}
    \sum_{l=1}^{j-1} (\varphi\circ\sigma^l(\underline{x})-\alpha_{\varphi})=S_\varphi(\sigma(\underline{x}),\sigma^j(\underline{x}))
    \end{align}
    for all $j\in\N$.
	For $i=0$ and any $j\in\N$, from a similar discussion as in the proof of Lemma~\ref{lemm:finite-sum}, taking $\{n_k\}_{k\in\N}\subset\N$ (not necessarily unbounded) and $z^{(k)}\in B(\underline{x},\sigma^j(\underline{x}),n_k;2^{-k})$ such that
      \[
      \lim_{k\to\infty} \sum_{i=0}^{n_k-1} \left(\varphi\circ\sigma^i(\underline{z}^{(k)})-\alpha_{\varphi}\right)=S_\varphi(\underline{x},\sigma^{j}(\underline{x})),
      \]
      we have $\sigma(z^{(k)})\in B(\sigma(\underline{x}),\sigma^j(\underline{x}),n_k-1;2^{-k+1})$
      and thus
    \begin{align*}
S_\varphi(\underline{x},\sigma^j(\underline{x}))&=\lim_{j\to\infty} \left\{ (\varphi(\underline{z}^{(k)})-\alpha_{\varphi})+\sum_{l=1}^{n_k-1} (\varphi\circ\sigma^l(\underline{z}^{(k)})-\alpha_{\varphi})\right\}\\
    &\ge \lim_{j\to\infty} \left\{(\varphi(\underline{x})-\alpha_{\varphi})- L_\varphi 2^{-j}+S_{\varphi}(\sigma(\underline{x}),\sigma^j(\underline{x});2^{-j+1})\right\}\\
    &= \varphi(\underline{x})-\alpha_{\varphi}+S_{\varphi}(\sigma(\underline{x}),\sigma^j(\underline{x}))\\
    & =\sum_{l=0}^{j-1} (\varphi\circ\sigma^l(\underline{x})-\alpha_{\varphi}).
\end{align*}
Here we use \eqref{eqn:Mane_induction} in the last equality.
    The opposite inequality follows from the definition of $S_\varphi$, which completes the proof of $N_2\subset N_\varphi$.

    We now show that $N_{\varphi}\subset N_1\cup N_2$.
    Take $\underline{x}\in N_\varphi$ and assume that $\underline{x}\notin N_1\cup N_2$.
    From $\underline{x}\notin N_1$, there exist two positive integers $\tau_1<\tau_2$ such that $\sigma^{\tau_1}(\underline{x})=\sigma^{\tau_2}(\underline{x})$.
    Then $\underline{y}=\sigma^{\tau_1}(\underline{x})$ is $\tau:=\tau_2-\tau_1$ periodic and belongs to $N_\varphi$ by the invariance of $N_\varphi$.
    By Proposition~\ref{prop:action_minimizing_periodic},
    we deduce that $\underline{x}\in \sigma^{-\tau}(\mathscr{M}_\varphi^{\mathrm{p}})$ holds but this contradicts to $\underline{x}\notin N_2$.
\end{proof}

From Theorem~\ref{thm:mane-set}, we can construct an example where the Ma\~{n}\'{e} set is not closed.
Notice that any cylinder set contains an element in $N_2$.

\begin{example}[The non-closed Ma\~{n}\'{e} set]
\label{ex:nonclosed}
Set
$\varphi(\underline{x})=g(x,y)=(x-y)^2+x^2\in \mathscr{H}.$
    Then $N_{\varphi}$ is not closed.
     It is easy to see $g^*:=\min_{a\in[0,1]} g(a,a)=0$ and $\mathrm{m}_g=\{0\}$.
         If $N_\varphi$ is closed, by Theorem~\ref{thm:mane-set},
    it holds that $N_{\varphi}=X$.
    On the other hand, it holds that $1^\infty\notin N_g$
    since the right-hand side of \eqref{eqn:semi_static} with any $0\le i<j$ 
    is $S_g(1^\infty,1^\infty)\le g(1,1)-g^*=1$
    but the left-hand side of \eqref{eqn:semi_static} with any $0\le i<j$ such that $j-i>1$ becomes
    \[
    \sum_{k=i}^{j-1}(g(1,1)-g^*)=j-i>1.
    \]
\end{example}

\begin{remark}
    Since $S_\varphi$ is lower semi-continuous (Proposition 2.7 in \cite{KMS25}), we see that if $S_\varphi$ is upper semi-continuous, then the continuity of $S_\varphi$ holds, and thus $N_\varphi$ is closed (even if $\varphi \in \mathcal{H}$ is not assumed).
    We remark that, in the setting of the Euler-Lagrange flow, the Ma\~{n}\'{e} set is always closed since the corresponding Ma\~{n}\'{e} potential is continuous.
\end{remark}

\subsection{Calibrated subactions and an equivalence relation on the Aubry set}\label{sec:calibrated_equivalent}
This subsection describes the equivalence class and its properties for the Aubry set.
Firstly, we introduce the definition of \textit{ calibrated subaction} and refer to the results in \cite{KMS25}, which are analogous to Theorem 4.1 in \cite{BLL13}.

\begin{definition}[Calibrated subaction]
    For a subaction $u$ of $\varphi$, it is called {\it calibrated} if 
    \begin{align*}
        \min_{\sigma(\underline{y})=\underline{x}}(\varphi(\underline{y})+u(\underline{y}))=u(\underline{x})+\alpha_\varphi
    \end{align*}
    for every $\underline{x} \in X$.
\end{definition}
\begin{theorem}[Part of Main Theorem~$2$ in \cite{KMS25}]
    \label{property_Peierl}
    For any $\underline{x}\in\Omega_\varphi$, the map $X\ni \underline{y}\mapsto H_\varphi(\underline{x},\underline{y})$ is a Lipschitz calibrated subaction of $\varphi$.
\end{theorem}

Calibrated subactions are powerful technical tools for ergodic optimization and closely related to the equivalence classes of the relation defined in Theorem~\ref{theorem:KMS25}, which we will see below in this subsection.
Before investigating their connection, we give the following formula.
\begin{theorem}
\label{Peierl's_barrier_inf}
    Every calibrated subaction $u$ of $\varphi$ satisfies 
    \begin{align}
        u(\underline{y})=\inf_{\underline{x}\in \Omega_\varphi}\left( H_\varphi(\underline{x},\underline{y})+u(\underline{x}) \right)
        \label{local_global}
    \end{align}
    for every $\underline{y}\in {X}$.
\end{theorem}
The case where the base space is a symbolic dynamical system with a finite alphabet was treated in \cite[Theorem 4.7]{BLL13}, in which a similar formula was given for a Lipschitz function. 
In the case of the full shift with $[0,1]$, a similar formula was established in \cite[Proposition 8]{LMST} for Lipschitz potentials depending only on the first two coordinates and its $[0,1]$-backward calibrated subactions.

Theorem~\ref{Peierl's_barrier_inf} indicates that, even if we know the values of a calibrated subaction $u$ only on $\Omega_{\varphi}$, $u$ can be consistently extended from $\Omega_{\varphi}$ to the entire space $X$ via the Peierl’s barrier, which encodes the minimal cost to reach $\underline{y}$ from $\Omega_{\varphi}$. 
Thus, $u$ is not merely a tool for analyzing minimizing orbits, but also a means of transmitting the structure of the Aubry set throughout $X$. 
In this sense, the Aubry set serves as the ``core" of optimal dynamics, and the calibrated subaction $u$ encodes its influence globally.

In order to show Theorem~\ref{Peierl's_barrier_inf}, we shall give the following lemma.
\begin{lemma}\label{accumuration_points_calibrated}
    Let $u: X\rightarrow \mathbb{R}$ be a calibrated subaction of $\varphi$ and $\underline{y}^{(0)}\in {X}$.
    For each $n\geq 1$ let $\underline{y}^{(-n)}$ be a preimage of $\underline{y}^{(-n+1)}$ such that
    \begin{align*}
        u(\underline{y}^{(-n+1)})=\varphi(\underline{y}^{(-n)})+u(\underline{y}^{(-n)})-\alpha_\varphi.
    \end{align*}
    Then any accumulation point of $\{\underline{y}^{(-n)}\}$ belongs to $\Omega_\varphi$.
\end{lemma}
\begin{proof}
    Let us consider a convergent subsequence $\{\underline{y}^{(-n_k)}\}$ of $\{\underline{y}^{(-n)}\}$ such that $\displaystyle{\lim_{k\to\infty}\underline{y}^{(-n_k)}=\underline{x} \in X}$.
    By \eqref{eq:mane_calibrated}, we have
    \begin{align*}
        S_\varphi(\underline{x},\underline{x})\geq u(\underline{x})-u(\underline{x})=0.
    \end{align*}
    Let $\theta>0$. Take $\varepsilon>0$ satisfying
    \begin{align*}
        S_\varphi(\underline{x},\underline{x})<\widehat{S}_\varphi(\underline{x},\underline{x};\varepsilon)+\theta.
    \end{align*}
    From the uniform continuity of $u$, we can take $\delta \in (0,\varepsilon)$ such that
    \begin{align}
    \label{eq:u_condition}
    \text{
    $|u(\underline{z})-u(\underline{z}')|<\theta$ if $d(\underline{z},\underline{z}')<\delta$.
    }
    \end{align}
    We can then choose $k\geq 1$ large enough to ensure
    \[
    d(\underline{y}^{(-n_k)},\underline{x})< \frac{\delta}{2}, \ \text{and} \ d(\underline{y}^{(-n_{k+1})},\underline{x})< \frac{\delta}{2}.
    \]
    This implies
    \begin{align}
    \label{eq:y_delta_condition}
    d(\underline{y}^{(-n_k)},\underline{y}^{(-n_{k+1})})<\delta
    \end{align}
    Moreover, we obtain
    \[
        \sigma^{n_{k+1}-n_k}(\underline{y}^{(-n_{k+1})})=\underline{y}^{(-n_k)}
    \]
    since $\underline{y}^{(-n)}$ is a preimage of $\underline{y}^{(-n+1)}$ for all $n\ge 1$.
    From the definition of $\underline{y}^{(-n)}$, $\eqref{eq:u_condition}$, and $\eqref{eq:y_delta_condition}$, it follows that
    \begin{align*}
        \sum_{i=0}^{n_{k+1}-n_k-1}(\varphi\circ\sigma^i(\underline{y}^{(-n_{k+1})})-\alpha_{\varphi})
        =u(\underline{y}^{(-n_{k})})-u(\underline{y}^{(-n_{k+1})})
        <\theta.
    \end{align*}
    Hence we have
    \begin{align*}
        S_\varphi(\underline{x},\underline{x})<2\theta.
    \end{align*}
    Since $\theta>0$ is arbitrary, we have $S_\varphi(\underline{x},\underline{x})\leq 0$, which completes the proof.
\end{proof}

\begin{proof}[Proof of Theorem \ref{Peierl's_barrier_inf}]
    Take $\underline{y}\in {X}$ and a sequence $\{\underline{y}^{(-n)}\}$ as in Lemma \ref{accumuration_points_calibrated} with $\underline{y}^{(0)}:=\underline{y}$ such that the subsequence $\{\underline{y}^{(n_k)}\}$ converges to some $\underline{x}\in\Omega_\varphi$.
    By the definition of $\underline{y}^{(-n)}$, we have
    \[
    \sum_{i=0}^{n_{k}-1}\left(\varphi\circ\sigma^i(\underline{y}^{(-n_{k})})-\alpha_{\varphi} \right)
    =u(\underline{y})-u(\underline{y}^{(-n_k)})
    \]
    for all $k\geq 1$.
   Let $\theta>0$. Take $\varepsilon>0$ small enough such that
   \begin{align*}
       H_\varphi(\underline{x},\underline{y})<\widehat{H}_\varphi(\underline{x},\underline{y}; \varepsilon)+\theta.
   \end{align*}
   Take $\delta \in (0,\varepsilon)$ such that $|u(\underline{z})-u(\underline{z}')|<\theta$ if $d(\underline{z},\underline{z}')<\delta$.
   We can then choose $k\geq 1$ large enough to satisfy $d(\underline{y}^{(-n_k)}, \underline{x})<\delta$ and 
   \begin{align*}
       \widehat{H}_\varphi(\underline{x},\underline{y};\varepsilon)
       \leq \inf_{n\geq n_k}
       \left\{\sum_{i=0}^{n-1}(\varphi\circ\sigma^i(\underline{z})-\alpha_{\varphi}): \underline{z}\in B(\underline{x}, \underline{y}, n;\varepsilon) \right\}+\theta.
   \end{align*}
   Since $\underline{y}^{(-n_k)}\in B(\underline{x},\underline{y}, n_k; \varepsilon)$, we have
   \begin{align*}
       H_\varphi(\underline{x}, \underline{y})
       &\leq \sum_{i=0}^{n_{k}-1}(\varphi\circ\sigma^i(\underline{y}^{-n_{k}})-\alpha_{\varphi})+2\theta\\
       &=u(\underline{y})-u(\underline{y}^{(-n_k)})+2\theta\\
       &\leq u(\underline{y})-u(\underline{x})+3\theta.
   \end{align*}
   Since $\theta>0$ is arbitrary and $\underline{x}\in \Omega_\varphi$, we have
   \begin{align*}
       u(\underline{y})\geq \inf_{\underline{x}\in \Omega_\varphi}(H_\varphi(\underline{x}, \underline{y})+u(\underline{x})).
   \end{align*}
   On the other hand, combining the formula \eqref{eq:mane_calibrated} with the definitions of $H_{\varphi}$ and $S_{\varphi}$, we obtain
    \begin{align*}
        \inf_{\underline{x}\in \Omega_\varphi}(H_\varphi(\underline{x}, \underline{y})+u(\underline{x}))
        &\ge\inf_{\underline{x}\in \Omega_\varphi}(S_\varphi(\underline{x}, \underline{y})+u(\underline{x}))\\
        &\ge \inf_{\underline{x}\in \Omega_\varphi} (u(\underline{y})-u(\underline{x}) + u (\underline{x}))\\
        &\ge u(\underline{y}),
    \end{align*}
    which completes the proof.
\end{proof}

Furthermore, as an analogy of Theorem 4.8 in \cite{BLL13}, we observe that a calibrated subaction $u$ remains consistent on each equivalence class (often referred to as a connected component) within the Aubry set. 
Specifically, points in the same class differ only by a constant under $u$.

\begin{theorem}
\label{thm:relation_H_u}
    Let $u$ be a calibrated subaction of $\varphi$.
    If $\underline{x},\underline{z}\in \Omega_\varphi$ satisfy the relation $\underline{x}\sim_\varphi \underline{z}$ (given in Theorem \ref{theorem:KMS25}(ii)), then we have
    \begin{align*}
        H_\varphi(\underline{x}, \underline{y})+u(\underline{x})=H_\varphi(\underline{z},\underline{y})+u(\underline{z})
    \end{align*}
    for every $\underline{y}\in X$.
\end{theorem}
\begin{proof}
    From the triangle inequality of the Peierl's barrier (Lemma 2.14 in \cite{KMS25}), we have
    \begin{align}
    H_\varphi(\underline{x},\underline{y})\leq H_\varphi(\underline{x},\underline{z})+H_\varphi(\underline{z},\underline{y})
    \label{eq:triangle}
    \end{align}
    for every $\underline{x},\underline{y},\underline{z}\in X$.
    Using \eqref{eq:triangle} and Theorem \ref{Peierl's_barrier_inf}, we obtain
    \begin{align*}
        H_\varphi(\underline{x},\underline{y})+u(\underline{x})
        &\leq H_\varphi(\underline{x},\underline{z})+H_\varphi(\underline{z}, \underline{y})+u(\underline{x})
        \\
        &\leq H_\varphi(\underline{x},\underline{z})+H_\varphi(\underline{z}, \underline{y})+u(\underline{x})+
        H_\varphi(\underline{z},\underline{x})+u(\underline{z})-u(\underline{x})
        \\
        &\leq H_\varphi(\underline{z},\underline{y})+u(\underline{z}). 
    \end{align*}
    Here we use $\underline{x}\sim_\varphi \underline{z}$ in the last inequality. Exchanging the roles of $\underline{x}$ and $\underline{z}$, we get the reverse inequality.
\end{proof}

In light of Theorem~\ref{thm:relation_H_u}, we see that Theorem \ref{Peierl's_barrier_inf} can be reformulated as
\begin{align}\label{eqn:reformulation_calibrated_subaction}
    u(\underline{y})=\inf_{\underline{x} \in \Omega_\varphi/\sim_\varphi} G([\underline{x}],\underline{y}),
\end{align}
where $[\underline{x}] = \{\underline{y} \in \Omega_{\varphi} \mid \underline{x} \sim_\varphi \underline{y}\}$
and $G([\underline{x}],\underline{y})=H_\varphi(\underline{z},\underline{y})+u(\underline{z})$ for some $\underline{z}\in [\underline{x}]$.
Thus, the ``core'' of optimization is not merely $\Omega_\varphi$ as a whole, but rather its equivalence classes, which serve as a more refined structural unit.
This leads to the study of the quotient Aubry set $\bar{\Omega}_\varphi=\Omega_\varphi/\sim_\varphi$
and, in Section~\ref{sec:equivalence-Aubry}, we will extremely investigate it for the case that $\varphi$ is a Lipschitz potentials depending only on the first two coordinates.
Note that, from the above reformulation, we obtain the uniqueness of calibrated subactions in Appendix~\ref{sec:appendix}, which is out of our main interests in the present paper but an important remark.

\section{Explicit formulae of action minimizing sets}
\label{sec:explicit-form}

This section is devoted to the proofs of Main Theorem \ref{thm:aubry} and Main Theorem \ref{thm:semi-static}.
Although these theorems assume that $h\in\mathscr{H}$, we actually consider a more general setting of Lipschitz potentials depending only on the first two coordinates including indifferentiable ones.

Set 
\begin{align}
    \mathcal{H}=\{ h \in \mathrm{Lip}([0,1]^2,\R) \mid \text{$h$ satisfies $(H_3)$ and $(H_4)$} \},
\end{align}
where $(H_3)$ and $(H_4)$ are given by
\begin{itemize}
    \item[$(H_3)$]  If $\xi_1 < \xi_2$ and $\eta_1 < \eta_2$,
        then \[
        h(\xi_1,\eta_1) + h(\xi_2,\eta_2) < h(\xi_1,\eta_2) + h(\xi_2,\eta_1).
        \]
    \item[$(H_4)$]
    If both $(x_{-1},x_0,x_{1})$ and $(x^{'}_{-1},x_0,x^{'}_{1})$
    with $(x_{-1},x_0,x_{1}) \neq (x^{'}_{-1},x_0,x^{'}_{1})$ are {\it{minimal}}, then \[(x_{-1}-x^{'}_{-1})(x_{1}-x^{'}_{1})<0.\]
\end{itemize}
Here, the definition of the terminology {\it minimal} is the following:
\begin{definition}[Minimal]
Fix $k,l \in \N_{0}$ with $k < l$ arbitrarily.
A finite sequence $\{x_i\}_{i=k}^{l}$ is said to be {\it minimal} if, for any $\{y_i\}_{i=k}^{l}$ with $y_k=x_k$ and $y_l = x_l$, we have:
\[
\sum_{i=k}^{l-1} h(x_i,x_{i+1}) \le \sum_{i=k}^{l-1} h(y_i,y_{i+1}).
\]
Moreover, an infinite sequence $\{x_i\}_{i \in \N_{0}}$ is said to be minimal if $\{x_i\}_{i=k}^{l}$ is minimal for any $k,l$ with $k < l$.
\end{definition}
The labels $(H_3)$ and $(H_4)$ come from \cite{Ban88}.
Hereafter, we always assume that
\[
\text{$\varphi(\underline{x})=h(x_0,x_1)$ for some $h \in \mathcal{H}$.}
\]

\begin{remark}
\label{rem:h}
\,
\begin{enumerate}[(i)]
    \item Note that $\mathscr{H} \subset \mathcal{H}$
    since the twist condition $D_1D_2h<0$ implies $(H_3)$ and $(H_4)$ (see Remark 4.4 in \cite{KMS25}).
    For the proofs of Main Theorem \ref{thm:aubry} and Main Theorem \ref{thm:semi-static} only $(H_3)$ and $(H_4)$ are required.
    Therefore, we can replace $\mathscr{H}$ with $\mathcal{H}$ in these theorems.
    \item As stated in Theorem~\ref{thm:criterion}(i), the optimal ergodic average $\alpha_\varphi$ with 
    $\varphi(\underline{x})=h(x_0,x_1)$ equals to $h^*=\min_{a\in[0,1]} h(a,a)$. Throughout Sections \ref{sec:explicit-form} and \ref{sec:equivalence-Aubry}, we frequently use this fact.
\end{enumerate}
\end{remark}
By Remark~\ref{rem:h}(ii) and our setting, we immediately get
\[
\sum_{i=0}^{n-1} \left(\varphi\circ\sigma^i(\underline{x})-\alpha_{\varphi} \right) = \sum_{i=0}^{n-1} (h(x_i,x_{i+1}) - h^\ast).
\]

\begin{lemma}
\label{lemm:period_aubry}
    If $\underline{x} \in \Omega_{\varphi}$ is periodic,
    then we have $\underline{x}= a^\infty$ for some $a \in \mathrm{m}_h$.
\end{lemma}
\begin{proof}
    From Theorem~\ref{thm:criterion} and Proposition~\ref{bounded},
    it suffices to show that
    \[
    \lim_{n \to \infty} \sum_{i=0}^{n-1}\left(\varphi\circ\sigma^i(\underline{x})-\alpha_{\varphi} \right)= \infty
    \]
    if $\underline{x}\in (\mathrm{m}_h)^{\mathbb{N}_0}$ and is periodic with a minimal period $k\geq 2$.

    We first show that, by induction, 
    \begin{align}
    \label{eq:per_indction}
    \sum_{i=0}^{l-1}\left(\varphi\circ\sigma^i(\underline{y})-\alpha_{\varphi} \right)=
    \sum_{i=0}^{l-1} (h(y_i,y_{i+1})-h^*)>0
    \end{align}
    for any periodic point $\underline{y}\in(\mathrm{m}_h)^{\mathbb{N}_0}$ with the minimal period $l$, where $l\geq 2$.
    Set $\underline{y}=(a_{k_1}a_{k_2})^\infty$ with $a_{k_1} \neq a_{k_2}$. By $(H_3)$,
    \[
    h(a_{k_1},a_{k_2}) + h(a_{k_2},a_{k_1}) > h(a_{k_1},a_{k_1}) + h(a_{k_2},a_{k_2}) = 2 \alpha_{\varphi},
    \]
    which implies that $\eqref{eq:per_indction}$ holds for $l=2$.
    We next assume that
    \[
    \sum_{i=0}^{l-2} (h(y_i,y_{i+1})-h^*)>0
    \]
    for any periodic point $\underline{y}\in (\mathrm{m}_h)^{\mathbb{N}_0}$ with the minimal period $l-1$.

    Let $\underline{y}=y_0y_1\ldots\in (\mathrm{m}_h)^{\mathbb{N}_0}$ be a periodic point with the minimal period $l\ge 2$.
    Then there exists $l'\in\{1,2,\cdots,l\}$
    such that $y_{l'}<y_{l'-1}$ and $y_{l'}\le y_{l'+1}$ since if not we have $y_i\ge y_{i-1}$ or $y_i> y_{i+1}$ for all $i\in\{1,\cdots, l\}$ and this contradicts to the fact that 
    $\underline{y}$ is periodic with a minimal period $l\ge 2$.
    Now we consider the case $\underline{y}=(a_1a_2\cdots a_l)^\infty\in (\mathrm{m}_h)^{\mathbb{N}_0}$ with $a_1<a_2$
    and $a_1\le a_l$.
     Since $a_{1} < a_{2}$ and $a_{1}\le a_{l}$, $(H_3)$ and the assumption for the induction imply
    \begin{align*}
    \label{eq:cn>cn-1}
    &h(a_{1},a_{2}) + h(a_{2},a_{3}) + \cdots + h(a_{l},a_{1})\\
    & \ge h(a_{1},a_{1}) + \underbrace{h(a_{l},a_{2})  + h(a_{2},a_{3}) + \cdots + h(a_{l-1},a_{l})}_{(l-1)-periodic}\\
    &> \alpha_{\varphi}+ (l-1)  \alpha_{\varphi} =l \alpha_{\varphi}.
    \end{align*}
    Note that the equality of the first inequality holds if and only if $a_1=a_l$.
    The other cases can be shown in a similar way, which completes the proof of \eqref{eq:per_indction} for any periodic point $\underline{y}\in (\mathrm{m}_h)^{\mathbb{N}_0}$ with a minimal period $l\ge 2$.

     Now, assume that $\underline{x}\in (\mathrm{m}_h)^{\mathbb{N}_0}$ is a periodic point with a minimal period $k$, where $k\ge 2$. Then we have
     \[
     c_k:=\sum_{i=0}^{k-1}(\varphi\circ\sigma^i(\underline{x})-\alpha_{\varphi} )=
     \sum_{i=0}^{k-1}(\varphi\circ\sigma^{jk+i}(\underline{x})-\alpha_{\varphi} )
     >0
     \]
     for any $j \in \N$.
     Thus, we compute
     \begin{align*}
     &\sum_{i=0}^{mk+r-1}(\varphi\circ\sigma^i(\underline{x})-\alpha_{\varphi} )\\
     &=
     \sum_{j=0}^{m-1} \sum_{i=0}^{k-1}(\varphi\circ \sigma^{jk+i}(\underline{x})-\alpha_{\varphi} )
     +
     \sum_{i=0}^{r-1}(\varphi\circ\sigma^{mk+i}(\underline{x})-\alpha_{\varphi} )\\
     &= c_k m+\sum_{i=0}^{r-1}(\varphi\circ\sigma^i(\underline{x})-\alpha_{\varphi} )
     \end{align*}
     for all $r=0,1,\ldots, k-1$.
     By taking $m \to \infty$, we obtain the desired result.
\end{proof}

\begin{lemma}[\cite{Ban88}, Lemma 4.5 \cite{KMS25}]
    \label{lemma:Ban33}
    Fix $n \in \N$.
    Then
    \[
    \sum_{i=0}^{n-1}(h(x_i,x_{i+1})-h^\ast) \ge 0
    \]
    for any $\underline{x} \in X(n)$, where
    \begin{align}
    \label{eq:xn-periodic}
    X(n) = \{ \underline{x} \in X \mid x_i=x_{n+i} \ \text{for all} \ i \in \N_{0} \}.
    \end{align}
    Moreover, the equality is true if and only if there exists $a \in \mathrm{m}$ satisfying ${x}_i = a$ for $i=0,\cdots,n-1$.
\end{lemma}
\begin{remark}\label{rmk:periodic}
 We can replace $X(n)$ with $Y(n):=\{\underline{x}\in X\mid x_0=x_n\}$ since $\sum_{i=0}^{n-1}(h(x_i,x_{i+1})-h^\ast)$ refers to only the first $n$ coordinates of $\underline{x}\in X$.
\end{remark}

{We now give a key estimate for ``connecting orbit" with two distinct initial/terminal words. It is worth to remark that the following lemma relies on the assumption $(H_4)$.}
\begin{lemma}[{Key estimate}]
\label{lemm:c-min}
If $a,b \in \mathrm{m}_h$ and $a \neq b$,
    \begin{align}
    \label{c-min}
    \begin{split}
    \tilde{c} := h(a,b)-h^\ast - H_{\varphi}(a^\infty,b^\infty)>0.
    \end{split}
    \end{align}
\end{lemma}

\begin{proof}
We first show that there exists $x^*\in[0,1]$ such that
\[
C=h(a,b)-h^*-\left(h(a,x^*)-h^*+h(x^*,b)-h^*\right)>0.
\]
Let $F(x)=h(a,x)-h^*+h(x,b)-h^*$ for $x\in[0,1]$.
Since $F$ is continuous on $[0,1]$, it has a minimum point $x^*\in[0,1]$. 
Hence it suffices to prove that $x^*\neq a,b$.
Assume that $x^*=a$. This implies that the word $aab$ is minimal.
By Lemma \ref{lemma:Ban33}, the word $aaa$ is also minimal and $(H_4)$ yields a contradiction. We can obtain $x^*\neq b$ in a similar way.
Therefore, since $x^*$ is a minimizer of $F$, we have
\begin{align*}
F(x^*)&=h(a,x^*)-h^*+h(x^*,b)-h^*\\
&<F(a)=h(a,a)-h^*+h(a,b)-h^*=h(a,b)-h^*,
\end{align*}
which implies that $C>0$.

Now we prove \eqref{c-min}.
Letting $\underline{y}^{(n)}=a^n x^*b^\infty$, we have
\begin{align*}
    \sum_{i=0}^{n} (h(y_i^{(n)}, y_{i+1}^{(n)})-h^*)=h(a,x^*)-h^*+h(x^*,b)-h^*.
\end{align*}
Fix $0<\theta <C/2$.
Take $\varepsilon>0$ and $N\geq 1$ such that 
\begin{align*}
    H_\varphi(a^\infty, b^\infty)&<\inf_{n\geq N}\left\{\sum_{i=0}^{n-1}(h(z_i, z_{i+1}-h^*)\mid \underline{z}\in B(a^\infty, b^\infty, n;\varepsilon)\right\}+\theta.
\end{align*}
Since we have $\underline{y}^{(n)}\in B(a^\infty, b^\infty, n+1;\varepsilon)$ for all sufficiently large $n$,
we obtain, enlarging $N$, if necessary
\begin{align*}
    H_\varphi(a^\infty, b^\infty)< \sum_{i=0}^{n}(h(y_i^{(n)}, y_{i+1}^{(n)})-h^*)+\theta
\end{align*}
for $n\geq N$.
Then we have
\begin{align*}
    &h(a,b)-h^*-H_\varphi(a^\infty, b^\infty)\\
    &\geq h(a,b)-h^*-\sum_{i=0}^{n}(h(y_i^{(n)}, y_{i+1}^{(n)})-h^*)-\theta\\
    &=h(a,b)-h^*-\left(h(a,x^*)-h^*+h(x^*,b)-h^*\right)-\theta>0.
\end{align*}
\end{proof}

By a similar way of the proof in Lemma \ref{lemm:c-min}, we immediately show the following:
\begin{corollary}
\label{co:c-n-min}
    Fix $n \in \N$. Let $a,b \in \mathrm{m}_h$ and $a \neq b$. For any $\{x_i\}_{i=0}^{n}$, $\{y_i\}_{i=0}^{n}$ with $x_0=y_n=a$ and $x_n=y_0=b$, both
    \[
    \sum_{i=0}^{n-1}(h(x_i,x_{i+1})-h^\ast) - H_{\varphi}(a^\infty,b^\infty),
    \text{and} \
    \sum_{i=0}^{n-1}(h(y_i,y_{i+1})-h^\ast) - H_{\varphi}(b^\infty,a^\infty)\]
    are positive.
\end{corollary}

Now we are ready to prove Main Theorem \ref{thm:aubry}.
\begin{proof}[Proof of Main Theorem \ref{thm:aubry}]
Note that
\[
\{a^\infty\mid a\in\mathrm{m}_h\}\subset \mathscr{M}_\varphi\subset\Omega_\varphi \subset \mathrm{m}_h^{\N_0}
\]
by Theorem~\ref{theorem:KMS25}(i) and Theorem~\ref{thm:criterion}(ii), (iii).
Since $\Omega_\varphi$ is $\sigma$-invariant, it suffices to show that $\underline{x}=ab\underline{x}'\notin\Omega_\varphi$ for any two distinct $a,b\in \mathrm{m}_h$ and $\underline{x}'\in (\mathrm{m}_h)^{\N_0}$.
Assume that $\underline{x}=ab\underline{x}' \in\Omega_\varphi$ for some $a,b\in\mathrm{m}_h$ and $\underline{x}'\in (\mathrm{m}_h)^{\N_0}$ with $a\neq b$.
Take arbitrary small $\theta>0$ and fix $\varepsilon>0$ sufficiently small  so that $|\widehat{S}_\varphi(\underline{x},\underline{x};\varepsilon)|<\theta$.
Since $\Omega_\varphi=\{\underline{y}\in X\mid H_\varphi(\underline{y},\underline{y})=0\}$ by Theorem \ref{theorem:H_finite}, we can take
$\{n_j\}_{j\in\N}\subset \N$ with $n_j\to\infty\ (j\to\infty)$ and $\{\underline{z}^{(j)}\}_{j\in\N}\subset X$ such that
\begin{align*}
    d_X(\underline{z}^{(j)},\underline{x})<\varepsilon, \ \text{and} \
    d_X(\sigma^{n_j}(\underline{z}^{(j)}),\underline{x})<\varepsilon
    \ \text{for all} \ j\in\N,
\end{align*}
and
\begin{align*}
    \lim_{j\to\infty} \sum_{i=0}^{n_j-1}(\varphi\circ\sigma^i(\underline{z}^{(j)})-\alpha_{\varphi})=\widehat{S}_\varphi(\underline{x},\underline{x};\varepsilon).
\end{align*}
In particular, $|z_0^{(j)}-a|<\varepsilon, |z_1^{(j)}-b|<2\varepsilon$ and $|z_{n_j}^{(j)}-a|<\varepsilon$ hold for all $j\in\N$.
Letting
\[
\underline{w}^{(j)}=(a b z_2^{(j)}\ldots z_{n_j-1}^{(j)})^\infty,
\]
we see that $w^{(j)}$ is $n_j$-periodic and not a fixed point.
Applying
Lipschitz continuity of $\varphi$, Lemma~\ref{lemm:c-min}, Corollary~\ref{co:c-n-min} and the triangle inequality \eqref{eq:triangle} of the Peierl’s barrier, we have
\begin{align*}
\sum_{i=0}^{n_j-1}(\varphi\circ\sigma^i(\underline{z}^{(j)})-\alpha_{\varphi})
&=\sum_{k=0}^{n_j-1} ( h(z_k^{(j)},z_{k+1}^{(j)})- h^*)\\
&\ge\sum_{k=0}^{n_j-1} ( h(w_k^{(j)},w_{k+1}^{(j)})- h^*) -(3+\sqrt{5})L_{h}\varepsilon\\
&= h(a,b) -h^\ast + \sum_{k=1}^{n_j-1} ( h(w_k^{(j)},w_{k+1}^{(j)})- h^*)-(3+\sqrt{5})L_{h}\varepsilon \\
&\ge \tilde{c}+H_{\varphi}(a^\infty,b^\infty) + H_{\varphi}(b^\infty,a^\infty) -(3+\sqrt{5})L_{h} \varepsilon.\\
&\ge \tilde{c} + H_{\varphi}(a^\infty,a^\infty)- (3+\sqrt{5}) {L_{h}} \varepsilon \\
&= \tilde{c} - (3+\sqrt{5}) {L_{h}} \varepsilon 
\end{align*}
Here $L_h$ is a Lipschitz constant of $h\in\mathcal{H}$.
Note that a potential depending only on the first two coordinates $\varphi$ is Lipschitz continuous on $X$ if and only if $h$ is Lipschitz continuous on $[0,1]^2$ (see Proposition 4.1 in \cite{KMS25}).
Therefore,
\[
0=\lim_{\varepsilon\to 0} \widehat{S}_\varphi(\underline{x},\underline{x};\varepsilon)
=\lim_{\varepsilon\to 0}\lim_{j\to\infty}
\sum_{i=0}^{n_j-1}(\varphi\circ\sigma^i(\underline{z}^{(j)})-\alpha_{\varphi})
\ge \tilde{c}>0,\]
which is a contradiction.
\end{proof}

Main Theorem~\ref{thm:aubry} gives complete descriptions of the Mather set $\mathscr{M}_{\varphi}$ and the Aubry set $\Omega_{\varphi}$ for $\varphi \in \mathcal{H}$.
Using Main Theorem~\ref{thm:aubry} and Theorem~\ref{thm:mane-set}, we also have the explicit formula of the Ma\~{n}\'{e} set of $\varphi\in\mathscr{H}$.
\begin{proof}[Proof of Main Theorem~\ref{thm:semi-static}]
Combining Main Theorem~\ref{thm:aubry} and Theorem~\ref{thm:mane-set}, we immediately obtain the desired result.
\end{proof}

\section{Equivalence on the Aubry set}
\label{sec:equivalence-Aubry}
As described in Section~\ref{sec:intro}, for a Lipschitz function $\varphi$ on $X$,
Theorem~\ref{theorem:KMS25}(ii) enables us to define an equivalence relation $\sim_\varphi$ on $\Omega_\varphi$ given by
\[
\underline{x}\sim_\varphi \underline{y}
\quad \text{if and only if}\quad 
H_\varphi(\underline{x},\underline{y})+H_\varphi(\underline{y},\underline{x})=0,
\]
where $\underline{x},\underline{y}\in \Omega_\varphi$.
In this section, we study equivalence classes of this equivalence relation for Lipschitz potentials depending only on the first two coordinates with $(H_3)$ and $(H_4)$
and
always assume that $\varphi$ is represented by
\[
\varphi(\underline{x})=h(x_0,x_1)
\]
with some $h\in \mathcal{H}$.

\subsection{Characterizations of equivalences}
From Main Theorem~\ref{thm:aubry}, it holds that
each element in $\Omega_{\varphi}$ is of the form $a^\infty$ with $a\in\mathrm{m}_h$.
Therefore, we want to characterize when $a^\infty\sim_\varphi b^\infty$ holds for $a,b\in\mathrm{m}_h$.
We will see that this equivalence relation is closely related with the connected components of $\mathrm{m}_h$.
{We emphasize that most of discussion in this subsection concerns ``connecting orbit" from $a^\infty$ to $b^\infty$ with $a,b\in \mathrm{m}_h$}.

We now introduce an equivalence relation on $\mathrm{m}_h$ derived from its connected components. For $a,b\in\mathrm{m}_h$, we write
\[
a\sim_{\mathrm{conn},h} b
\]
if $a$ and $b$ belong to the same connected component of $\mathrm{m}_h\subset [0,1]$.
It is clear that $\sim_{\mathrm{conn},h}$ is an equivalence relation in $\mathrm{m}_h$
, and
we denote the connected component of $\mathrm{m}_h$ containing $a\in\mathrm{m}_h$ by $C(a)$.
Then $\mathrm{m}_h$ can be expressed as
\[
\mathrm{m}_h  = \bigcup_{\theta\in\Theta} A_\theta,
\]
where $\{A_\theta\}_{\theta\in\Theta}$ is a family of disjoint closed intervals and $\Theta$ is an index set of connected components of $\mathrm{m}_h$.
Under the additional assumption \eqref{eq:add_condition}, as described in Theorem \ref{thm:equivalent_interval}, we show that both $a$ and $b$ belong to the same connected components of $\mathrm{m}_{h}$  if and only if $a^\infty \sim_{\varphi} b^\infty$.
We first prove that $a^\infty \sim_{\varphi} b^\infty$ implies $a\sim_{\mathrm{conn},h} b$, which holds without the additional assumption.
\begin{theorem}
\label{theorem:not_equivalent} 
    Let $a ,b \in \mathrm{m}_h$ be elements of $\mathrm{m}_h$ such that $a\nsim_{\mathrm{conn},h} b$, i.e., $a$ and $b$ do not belong to the same connected component of $\mathrm{m}_h$. Then $a^\infty \nsim_{\varphi} b^\infty$.
\end{theorem}

For the proof of Theorem~\ref{theorem:not_equivalent},
we employ the following result, which is a slightly 
extended version of Lemma 2.7 of \cite{Yu22}.

\begin{lemma}[Lemma 4.9 of \cite{KMS25}]
\label{lemma:Yu27}
    Set
    \[\phi(\delta)=\inf_{n \in \N} \phi(\delta;n)\]
    where
    \[\phi(\delta;n)=\inf \left\{\sum_{i=0}^{n-1}\left(\varphi\circ \sigma^i(\underline{x})-\alpha_\varphi\right)\mid \underline{x} \in X(n), \ \max_{0 \le i \le n-1} d_{\R}({x}_i, \mathrm{m}_h) \ge \delta \right\}
    \]
    and $X(n)$ is given by $\eqref{eq:xn-periodic}$.
    Then $\phi(\delta)>0$ if $\delta>0$.
\end{lemma}
\begin{proof}
See \cite{KMS25}. We need $(H_3)$ and $(H_4)$ for the proof.
\end{proof}

\begin{remark}
\label{rem:Yu27}
    Since $\varphi$ depends on the first two coordinates, only words of length $n+1$ affect the infimum in $\phi(\delta;n)$.
    Hence $\phi(\delta;n)$ can be written as the infimum of 
    \[
        \sum_{i=0}^{n-1}\left(\varphi\circ \sigma^i(\underline{x})-\alpha_\varphi\right)
    \]
    over all $\underline{x}\in X$ such that its first $n+1$ coordinates $x_0\cdots x_{n-1}x_n$ satisfy $x_0=x_n$ and $\displaystyle\max_{0 \le i \le n-1} d_{\R}({x}_i, \mathrm{m}_h) \ge \delta$.
\end{remark}

\begin{proof}[Proof of Theorem~\ref{theorem:not_equivalent}]
We show only the case of $\mathrm{m}_h=A_0 \cup A_1$, i.e., $a \in A_0$, $b \in A_1$ and
\[
    \max A_0<\min A_1.
\]
Set
\begin{align*}
\delta \in \left(0, \frac{1}{3}d_{\R}(A_0,A_1) \right), \ \text{and}
\ A_i(\varepsilon)=\{ x \in [0,1] \mid d_{\R}(x,A_i) \le \varepsilon\}
\end{align*}
for $\varepsilon>0,i=1,2$.
Let $\{\varepsilon_k\}_{k \in \N}$ be a monotone decreasing positive sequence with $\varepsilon_k \to 0$ as $k \to \infty$.
Take $\underline{z}^{(k)} \in B(a^\infty,b^\infty,n_k,\varepsilon_k)$ and $\underline{w}^{(k)} \in B(b^\infty,a^\infty,m_k,\varepsilon_k)$ 
such that
\begin{align*}
    \lim_{k\to \infty} \sum_{i=0}^{n_k-1} \left(h(z_i^{(k)},z_{i+1}^{(k)})-h^*\right)&=H_\varphi(a^\infty,b^\infty),\\
    \lim_{k\to \infty} \sum_{i=0}^{m_k-1} \left(h(w_i^{(k)},w_{i+1}^{(k)})-h^*\right)&=H_\varphi(b^\infty,a^\infty).
\end{align*}
From the assumption, $\delta$ is positive and we assume $\varepsilon_k < \delta$.
Set
\[
    \underline{v}^{(k)}=(a z_1^{(k)} \cdots z_{n_k-1}^{(k)}b w_1^{(k)} \cdots w_{m_k-1}^{(k)} )^\infty.
\]
In view of Remark \ref{rem:Yu27}, suppose that
$\{\underline{v}^{(k)}\}$ satisfies
\[
    \max_{0 \le i \le n_k+m_k-1} d_{\R}({v}_i^{(k)}, \mathrm{m}_h) \ge {\delta}
\]
for infinitely many $k \in \N$.
Then it follows from Lemma \ref{lemma:Yu27} that, for sufficiently large $k$ so that $4L_h \varepsilon_k<\frac{1}{2}\phi(\delta)$, we obtain
\begin{align*}
    &\sum_{i=0}^{n_k-1} (h(z_i^{(k)},z_{i+1}^{(k)})-h^\ast)
    +\sum_{i=0}^{m_k-1} (h(w_i^{(k)},w_{i+1}^{(k)})-h^\ast)\\
    &\ge \sum_{i=0}^{n_k+m_k-2} (h(v_i^{(k)},v_{i+1}^{(k)})-h^\ast) - 4L_{h}\varepsilon_k\\
    &\ge \phi(\delta)-4L_{h}\varepsilon_k > \frac{1}{2} \phi(\delta)>0.
\end{align*}
Otherwise, by taking a subsequence of $k$ if necessary, we see that $v^{(k)}$ consists of the elements in $A_0(\delta)$ or $A_1(\delta)$ for all $k\in\N_0$.
Moreover, there exists $\tilde{n}_k\in\{1,\ldots,n_k-1\}$ (resp. $\tilde{m}_k\in\{1,\ldots,m_k-1\}$) such that
\begin{align}
\label{eq:z-div}
     z_{\tilde{n}_k}^{(k)} \in A_0(\delta), \ z_{\tilde{n}_k+1}^{(k)} \in A_1(\delta)
\qquad\left(\text{resp. } 
    w_{\tilde{m}_k}^{(k)} \in A_1(\delta), \ w_{\tilde{m}_k+1}^{(k)} \in A_0(\delta)\right),
\end{align}
where $z_{n_k}^{(k)}:=b$ (resp. $w_{m_k}^{(k)}:=a$).

Hereafter, we assume that $h^\ast=0$ for simplicity. The case of $h^\ast \neq0$ is shown in a similar way.
Set 
\[
    c = \min_{x,w \in A_0(\delta), y,z \in A_1(\delta)} \left(h(x,y)+h(z,w) -h(x,w)-h(z,y)\right).
\]
It is easily seen that $(H_3)$ implies $c>0$ and that $c$ is well-defined since both $A_0$ and $A_1$ are closed sets and $A_0{(\delta)} \cap A_1{(\delta)} = \emptyset$.
By Lemma 4.5 of \cite{KMS25},
\[
    h(a z_1^{(k)} \cdots z_{\tilde{n}_k}^{(k)})+h(z_{\tilde{n}_k}^{(k)},w_{\tilde{m}_k+1}^{(k)})+ h(w_{\tilde{m}_k+1}^{(k)} \cdots w_{m_k-1}^{(k)} a ) \ge 0
\]
and
\[
    h(z_{\tilde{n}_k+1}^{(k)} \cdots z_{n_k-1}^{(k)}b w_1^{(k)} \cdots w_{\tilde{m}_k}^{(k)} )+ h(w_{\tilde{m}_k}^{(k)},z_{\tilde{n}_k+1}^{(k)} ) \ge 0,
\]
where
\[
    h(a_0a_1\cdots a_n) = \sum_{i=0}^{n-1} h(a_i,a_{i+1}).
\]
Thus we get
\begin{align*}
    &\sum_{i=0}^{n_k+m_k-1} \left(h(v_i^{(k)},v_{i+1}^{(k)})-h^*\right)\\
    &=h(a z_1^{(k)} \cdots z_{\tilde{n}_k}^{(k)})
    +h(z_{\tilde{n}_k}^{(k)},z_{\tilde{n}_k+1}^{(k)})+h(z_{\tilde{n}_k+1}^{(k)} \cdots z_{n_k-1}^{(k)}b w_1^{(k)} \cdots w_{\tilde{m}_k}^{(k)} )\\
    &\qquad+ h(w_{\tilde{m}_k}^{(k)},w_{\tilde{m}_k+1}^{(k)} )+h(w_{\tilde{m}_k+1}^{(k)} \cdots w_{m_k-1}^{(k)} a )\\
    &=\left\{h(a z_1^{(k)} \cdots z_{\tilde{n}_k}^{(k)})+h(z_{\tilde{n}_k}^{(k)},w_{\tilde{m}_k+1}^{(k)})+ h(w_{\tilde{m}_k+1}^{(k)} \cdots w_{m_k-1}^{(k)} a )\right\}\\
    &\quad +\left\{h(z_{\tilde{n}_k+1}^{(k)} \cdots z_{n_k-1}^{(k)}b w_1^{(k)} \cdots w_{\tilde{m}_k}^{(k)} )+ h(w_{\tilde{m}_k}^{(k)},z_{\tilde{n}_k+1}^{(k)} )\right\}\\
    &\quad +\left\{h(z_{\tilde{n}_k}^{(k)},z_{\tilde{n}_k+1}^{(k)}) + h(w_{\tilde{m}_k}^{(k)},w_{\tilde{m}_k+1}^{(k)}) - h(z_{\tilde{n}_k}^{(k)},w_{\tilde{m}_k+1}^{(k)}) - h(w_{\tilde{m}_k}^{(k)},z_{\tilde{n}_k+1}^{(k)} )\right\}\\
    &\ge 0+0+c>0,
\end{align*}
which is the desired result.
\end{proof}

\begin{remark}
We give a remark for $\eqref{eq:z-div}$.
In fact, we can divide $\underline{v}^{(k)}$ into only three parts as follows:
    \begin{align}
    \label{eq:div_v}
    ((a z_1^{(k)} \cdots z_{\tilde{n}_k}^{(k)})(z_{\tilde{n}_k+1}^{(k)} \cdots z_{n_k-1}^{(k)}b w_1^{(k)} \cdots w_{\tilde{m}_k}^{(k)} ) (w_{\tilde{m}_k+1}^{(k)} \cdots w_{m_k-1}^{(k)} ))^\infty
\end{align}
where each element of the first and third parts is in $A_0({\delta})$ and the others are in $A_1({\delta})$.
There is no need to consider the case that $\eqref{eq:div_v}$ is divided into five or more parts.
The proof is as follows:
Assume that the first and second parts of $\eqref{eq:div_v}$ can be divided into
\begin{align}
\label{eq:v-5part}
(a z_1^{(k)} \cdots z_{\tilde{n}_k}^{(k)})
(z_{\tilde{n}_k+1}^{(k)})
(z_{\tilde{n}_k+2}^{(k)} \cdots z_{\tilde{n}_l}^{(k)})
(z_{\tilde{n}_l+1}^{(k)} \cdots z_{n_k-1}^{(k)}b w_1^{(k)} \cdots w_{\tilde{m}_k}^{(k)} ) \cdots
\end{align}
such that
both $(a z_1^{(k)} \cdots z_{\tilde{n}_k}^{(k)})$ and $(z_{\tilde{n}_k+2}^{(k)} \cdots z_{\tilde{n}_l}^{(k)})$ consist of the elements in $A_0(\delta)$ and the rest in $A_1(\delta)$.
Since it holds that
\[
h(z_{\tilde{n}_k}^{(k)},z_{\tilde{n}_k+1}^{(k)} )
+h(z_{\tilde{n}_k+1}^{(k)},z_{\tilde{n}_k+2}^{(k)} ) -2h^\ast
> h(z_{\tilde{n}_k}^{(k)}, z_{\tilde{n}_k+2}^{(k)} ) - h^\ast
\]
by $(H_3)$, we can construct a sequence whose sum is smaller than $\eqref{eq:v-5part}$.
The other cases are proven in a similar way.
\end{remark}

Next, we show the converse statement of Theorem \ref{theorem:not_equivalent} under the additional assumption.
\begin{theorem}\label{thm:equivalent_interval}
Let $a,b$ be elements of $\mathrm{m}_h$ such that $a\sim_{\mathrm{conn},h} b$, i.e., both of $a$ and $b$ belong to the same connected component of $\mathrm{m}_h$.
Suppose that $h$ satisfies the following:
\begin{align}
\label{eq:add_condition}
    \lim_{\delta \to +0} L_1(z,z,\delta) + L_2(z,z,\delta) =0 \ \text{for any} \ z \in C(a)
\end{align}
where
\begin{align*}
    L_1(x,y,\delta)&=\frac{h(x+\delta,y)-h(x,y)}{\delta} \ \text{and} \\
    L_2(x,y,\delta)&=\frac{h(x,y+\delta)-h(x,y)}{\delta}.
\end{align*}
Then $a^\infty \sim_{\varphi} b^\infty$.
\end{theorem}

\begin{proof}[Proof]
Without loss of generality we may assume $a<b$. Let $k\in\N$.
    Set $\{x_i^{(k)}\}_{i=0}^{2{k}}$ by $x_i^{(k)} = a+\frac{i(b-a)}{k}$ for $i\le k$ and $x_{i}^{(k)} = b-\frac{(i-k)(b-a)}{k}$ otherwise.

    Firstly, we show the following.
    \begin{align}
    \label{eq:sum0}
    \lim_{k \to \infty} \sum_{i=0}^{2k-1} (h(x_i^{(k)},x_{i+1}^{(k)}) - h^\ast) =0.
    \end{align}
    Set
    \[
    L(\delta)= \max_{a' \in \mathrm{m}_h}L_1(a',a',\delta) + L_2(a',a',\delta).
    \]
    The right-hand side is well defined since $\mathrm{m}_h$ is a closed set and $L_i(x,x,\delta)$ is continuous with respect to $x$ for $i=1,2$, and $L(\delta) \to 0$ as $\delta \to 0$.
    Thus we get
    \begin{align*}
     &\sum_{i=0}^{2k-1} (h(x_i^{(k)},x_{i+1}^{(k)}) - h^\ast )\\
    &=\sum_{i=0}^{k-1} (h(x_i^{(k)},x_{i+1}^{(k)}) - h(x_i^{(k)},x_i^{(k)}))+\sum_{i=0}^{k-1} (h(x_{i+1}^{(k)},x_{i}^{(k)}) - h(x_i^{(k)},x_i^{(k)}))\\
    &=\frac{b-a}{k}\sum_{i=0}^{k-1} \left(\frac{h(x_i^{(k)},x_{i+1}^{(k)}) - h(x_i^{(k)},x_i^{(k)})}{(b-a)/k}+\frac{h(x_{i+1}^{(k)},x_{i}^{(k)}) - h(x_i^{(k)},x_i^{(k)})}{(b-a)/k}\right)\\
     &\le \frac{b-a}{k}\sum_{i=0}^{k-1} L\left( \frac{(b-a)}{k} \right)=(b-a) L\left( \frac{(b-a)}{k} \right) \to 0 \quad (k \to \infty).
    \end{align*}
    Set
    \[
    \varepsilon_k=\sum_{i=0}^{2k-1} (h(x_i^{(k)},x_{i+1}^{(k)}) - h^\ast).
    \]
    The above remarks imply $\varepsilon_k \to 0$ as $k \to \infty$.
    Fix $\varepsilon>0$ arbitrarily.
    Set 
    \begin{align*}
    \underline{v}^{(k)}&=a^k x_0^{(k)}\cdots x_{k-1}^{(k)}b^\infty,  \ \text{and} \ \\
    \underline{w}^{(k)}&=b^k y_0^{(k)}\cdots y_{k-1}^{(k)}a^\infty
    \end{align*}
    where $y_i^{(k)}= x_{i+k}^{(k)}$.
    Taking sufficiently large $K$, we have
    \[
    \underline{v}^{(k)} \in B(a^\infty,b^\infty,2{k},\varepsilon), \quad
    \underline{w}^{(k)} \in B(b^\infty,a^\infty,2{k},\varepsilon)
    \]
    for any $k \ge K$,
    and
    \[
    \lim_{k \to \infty}
    \sum_{i=0}^{2k-1}(\varphi\circ \sigma^i(\underline{v}^{(k)})-\alpha_{\varphi})+ \sum_{i=0}^{2k-1}(\varphi\circ \sigma^i(\underline{w}^{(k)})-\alpha_{\varphi})
    = \lim_{k \to \infty} \varepsilon_k=0.
    \]
    It implies that
    \[
    H_{\varphi}(a^\infty,b^\infty) + H_{\varphi}(b^\infty,a^\infty)=0,
    \]
    which is the desired result.
\end{proof}

\begin{remark}
\label{remark:differential}
We give two remarks for differentiability of $h\in \mathcal{H}$.
\begin{enumerate}[(i)]
    \item 
    We impose the additional condition $\eqref{eq:add_condition}$ to exclude the following example.
    Set
    \[
    h(x,y)=\frac{1}{2}|x-y| + \sqrt{1+(x-y)^2}.
    \]
    Then 
    $\mathrm{m}_h=[0,1]$.
    Indeed, $h(x,y)$ satisfies both $(H_3)$ and $(H_4)$, and it holds that $a^\infty \nsim b ^\infty$ for any $a,b \in \mathrm{m}_h$.
    We will discuss the details in Section \ref{subsec:Smallness of the quotient Aubry set}.

    \item Assume that $(a,b)\subset \mathrm{m}_h$ and $h$ is differentiable at each point in $\Lambda=\{(x,y)\in(a,b)^2\mid x=y\}$. 
    Then, for $x\in(a,b)\subset \mathrm{m}_h$, we have
    \[
	0=\frac{d}{dx} h(x,x)=D_1h(x,x)+D_2 h(x,x).
    \]
    Therefore, the condition
    \[
    \lim_{\delta\to +0}\frac{h(x+\delta,x)-h(x,x)}{\delta}+\frac{h(x,x+\delta)-h(x,x)}{\delta}=0
    \]
    always holds if $h$ is differentiable at each point in $\Lambda$. Note that the Lipschitz continuity of $h$ and Rademacher's theorem imply that $h$ is differentiable almost everywhere in $\R^2$,
    but $D_2h(x,x)$ (or $D_1h(x,x)$) may not be defined for every $x\in [0,1]$, e.g., $h(x,y)=|x-y|+x(x-y)$.
\end{enumerate}
\end{remark}

It is worth mentioning that, under a stronger condition, we can obtain an explicit formula for Peierl's barrier for two fixed points associated with two distinct points in a connected component of $\mathrm{m}_h$.
\begin{proposition}
\label{prop:explicit_estimate}
    Let $I_{a,b}$ be a closed interval between $a,b\in[0,1]$ ($a\neq b$).
    Assume that $h$ is $C^1$ near $\Lambda=\{(x,y)\in I_{a,b}^2\mid x=y\}$.
    If $I_{a,b}\subset \mathrm{m}_h$, then
    \[
    H(a^\infty,b^\infty)= \int_a^b D_2h(x,x) dx
    =- \int_a^b D_1h(x,x) dx=-H(b^\infty,a^\infty).
    \]
\end{proposition}

\begin{proof}
	We only consider the case $a<b$.
	Let $x\in I_{a,b}$. By $(H_3)$, for any $x_0=a<x_1<x_2=b$, we get
    \begin{align*}
    h(x_0,x_2) +h(x_1,x_1)>h(x_0,x_1)+h(x_1,x_2),
    \end{align*}
    that is,
    \begin{align*}
    h(x_0,x_2) -h(x_0,x_0)>h(x_0,x_1)-h(x_0,x_0)+h(x_1,x_2)-h(x_1,x_1).
    \end{align*}
    By a similar computation, for any finite partition
    \[
    P=\{x_0=a<x_1<\ldots<x_n=b\}
    \]
    of $I_{a,b}$ and any refinement
    \[
    Q=\{x_0=a<x'_1<\ldots<x'_m=b\}
    \]
    of $P$ with $n<m$, we have
    \begin{align*}
    	\sum_{i=0}^{n-1} (h(x_i,x_{i+1}) -h(x_i,x_i))>
	\sum_{j=1}^{m-1} (h(x'_j,x'_{j+1})-h(x'_j,x'_j)).
    \end{align*}
    This implies
    \[
    	\inf_{P\in \mathcal{P}} \sum_{i=0}^{n-1} (h(x_i,x_{i+1}) -h(x_i,x_i))=\lim_{\delta\to 0} \inf_{P\in \mathcal{P}_\delta}\sum_{i=0}^{n-1} (h(x_i,x_{i+1}) -h(x_i,x_i)),
    \]
    where $\mathcal{P}$ stands for the set of  finite partitions
    \[
    P=\{x_0=a<x_1<\ldots<x_n=b\}
    \]
    of $I_{a,b}$
    and $\mathcal{P}_\delta\subset \mathcal{P}$ denotes the subset of partitions satisfying $\max_{i} |x_{i+1}-x_i|\le \delta$. 
    Fix sufficiently small $\theta>0$. By the definition of the Riemann integral, there exist $\delta>0$, $P\in\mathcal{P}_\delta$, and $t_i\in[x_i,x_{i+1}]\ (i=0,\ldots,n-1)$ such that
    \[
    	\left|\int_a^b D_2h(x,x) dx-\sum_{i=0}^{n-1} D_2h(t_i,t_i)(x_{i+1}-x_i)\right|<\theta.
    \]
    Moreover, since $h(x,y)$ is $C^1$ near $\Lambda$, by the mean value theorem and the uniform continuity of $D_2h$, there exist $0<\delta'<\delta$ and $\xi_i\in [x_i,x_{i+1}]\ (i=0,\ldots, n-1)$ such that
    \[h(x_i,x_{i+1})-h(x_i,x_i)=D_2 h(x_i,\xi_i)(x_{i+1}-x_i),
    \]
    \[
    \max_{i} |x_{i+1}-x_i|\le \delta',
    \]
    and
    \[
    |D_2 h(x_i,\xi_i)-D_2h(t_i,t_i)|<\theta.
    \]
    Therefore, we have
    \begin{align*}
    	&\left|\left(\sum_{i=0}^{n-1} h(x_i,x_{i+1}) -h(x_i,x_i)\right)-\int_{a}^b D_2 h(x,x) dx\right|\\
	&\le
	\sum_{i=0}^{n-1}|D_2 h(x_i,\xi_i)-D_2h(t_i,t_i)|(x_{i+1}-x_i)+\theta\\
	&\le \theta(b-a)+\theta=(b-a+1)\theta,
    \end{align*}
    which implies
    \[
    	\inf_{P\in \mathcal{P}} \sum_{i=0}^{n-1} (h(x_i,x_{i+1}) -h(x_i,x_i))=
	\int_{a}^b D_2 h(x,x) dx,
    \]
    where $P=\{x_0=a<x_1<\ldots<x_n=b\}$.
    Thus, we obtain
    \begin{align*}
    	\inf_{P\in \mathcal{P}} \sum_{i=0}^{n-1} (h(x_i,x_{i+1}) -h^\ast) &\ge
    	\inf_{P\in \mathcal{P}} \sum_{i=0}^{n-1} (h(x_i,x_{i+1}) -h(x_i,x_i))\\
	&\left(=\int_{a}^b D_2 h(x,x) dx\right)
    \end{align*}
    and the equality holds if and only if $I_{a,b}\subset \mathrm{m}$.

Now, we show that
    \[
    \inf_{P\in \mathcal{P}} \sum_{i=0}^{n-1} (h(x_i,x_{i+1}) -h^\ast) =H_\varphi(a^\infty,b^\infty).
    \]
    Let $P_j=\{x_0^{(j)}=a<x_1^{(j)}<\ldots<x_{n_j}^{(j)}\}\in\mathcal{P}\ (j\in\N)$ such that
    \[
    \lim_{j\to+\infty}\sum_{i=0}^{n_j-1} (h(x_i,x_{i+1}) -h^\ast)=\inf_{P\in \mathcal{P}} \sum_{i=0}^{n-1} (h(x_i,x_{i+1}) -h^\ast).
    \]
    Letting
    $\underline{w}^{(j)}=a^j x_0^{(j)}x_1^{(j)}\ldots x_{n_j}^{(j)}b^\infty\ (j\in\N)$,
    we have $\underline{w}^{(j)}\in B(a^\infty,b^\infty,j+n_j,2^{-j})$
    and thus
    \[
    \lim_{j\to +\infty} \widehat{H}_\varphi(a^\infty,b^\infty;2^{-j})\le \lim_{j\to+\infty}\sum_{i=0}^{n_j-1} (h(x_i,x_{i+1}) -h^\ast),
    \]
    which implies that
    \[
    H_\varphi(a^\infty,b^\infty)\le \int_a^b D_2h(x,x) dx
    \]
    if $I_{a,b}\subset \mathrm{m}_h$.
    Similarly, considering $\inf_{P\in \mathcal{P}} \sum_{i=0}^{n-1} (h(x_{i+1},x_{i}) -h(x_i,x_i))$,
    we have
    \[
    	H_\varphi(b^\infty,a^\infty)\le  \int_a^b D_1h(x,x) dx
	\quad\left(=\int_b^a D_2h(x,x) dx
\right)
    \]
    if $I_{a,b}\subset \mathrm{m}_h$ (see also Remark~\ref{remark:differential}(ii)).
    Therefore, it holds that
    \begin{align*}
    	0&\ge \left(H_\varphi(a^\infty,b^\infty)-\int_a^b D_2h(x,x) dx\right)+\left(H_\varphi(b^\infty,a^\infty)-\int_b^a D_2h(x,x) dx\right)\\
	&\ge H_\varphi(a^\infty,b^\infty)+H_\varphi(b^\infty,a^\infty)\ge0,
    \end{align*}
    which implies
    \[
    	H_\varphi(a^\infty,b^\infty)=\int_a^b D_2h(x,x) dx,\qquad 
	H_\varphi(b^\infty,a^\infty)=\int_b^a D_2h(x,x) dx
    \]
    if $I_{a,b}\subset \mathrm{m}_h$.
    \end{proof}

\subsection{Smallness of the quotient Aubry set}
\label{subsec:Smallness of the quotient Aubry set}
We now analyze a more topological aspect of the Aubry set and the equivalence relation $\sim_{\varphi}$ on it.
Note that, as mentioned at the beginning of this section, we always assume that $\varphi$ depends on the first two coordinates, and satisfies $(H_3)$ and $(H_4)$, i.e., $\varphi(\underline{x})=h(x_0,x_1)$ for some $h \in \mathcal{H}$.
The following proposition is easy.
\begin{proposition}[Main Theorem ~\ref{thm:Aubry_small}(i)]\label{prop:iso_Aubry_m}
    The Aubry set $\Omega_{\varphi}$ is isometric to $\mathrm{m}_h\subset\R$.
\end{proposition}
\begin{proof}
    By Main Theorem~\ref{thm:aubry}, it holds that
    \[
    \Omega_{\varphi}=\{a^\infty\mid a\in \mathrm{m}_h\}.
    \]
    Consider a map $\xi:\mathrm{m}_h\to \Omega_{\varphi}$ given by $\xi(a)=a^\infty$. Clearly, this map is surjective. Moreover, we have  \[
    d_X(\xi(a),\xi(b))=d_X(a^\infty,b^\infty)=\sum_{i=0}^\infty \frac{|a-b|}{2^{i+1}}=|a-b|=d_{\R}(a,b),
    \]
    which implies that $\xi$ is a bijective isometry from $(\mathrm{m}_h,d_\R)$ to $(\Omega_\varphi,d_X)$.
\end{proof}

We next consider the quotient Aubry set for $h\in\mathcal{H}$.
As stated in Section~\ref{sec:intro},
\[
\delta_\varphi(\underline{x},\underline{y})=H_\varphi(\underline{x},\underline{y})+H_\varphi(\underline{y},\underline{x})
\]
is a pseudo-metric on $\Omega_\varphi$
since it is symmetric and both of its non-negativity and its triangle inequality follow from the triangle inequality \eqref{eq:triangle} of the Peierl's barrier $H_\varphi$ and the identity $\Omega_\varphi=\{\underline{x}\in X \mid H_\varphi(\underline{x},\underline{x})=0\}$ (Theorem~\ref{theorem:KMS25}).
Therefore, $\delta_\varphi$ induces a metric on the quotient Aubry set $\bar{\Omega}_\varphi:=\Omega_\varphi/\sim_\varphi$.
\begin{proposition}
    If $\mathrm{m}_h$ is totally disconnected,
    then $\Omega_{\varphi}$ is homeomorphic to $\bar{\Omega}_{\varphi}$.
\end{proposition}
\begin{proof}
    By Theorem~\ref{theorem:not_equivalent}, 
    if $\pi(a^\infty)=\pi(b^\infty)$ holds (equivalently, $a^\infty\sim_{\varphi} b^\infty$ holds) 
    then $a\sim_{\mathrm{conn},h}b$, and the total disconnectedness of $\mathrm{m}_h$ implies $a=b$.
    Thus the natural projection $\pi:\Omega_{\varphi} \to \bar{\Omega}_{\varphi}$ is injective. The surjectivity of $\pi$ is trivial.
    Moreover, by Theorem~\ref{theorem:H_finite}, it is easy to see that $\pi$ is continuous
    since
    \[
	\delta_\varphi(\pi(\underline{x}),\pi(\underline{y}))
	=H_\varphi(\underline{x},\underline{y})+H_\varphi(\underline{y},\underline{x})\le 2L_{\varphi} d(\underline{x},\underline{y}). 
	\]
    Now we consider the continuity of $\pi^{-1}$.
    Assume that there exists $\{y_n\}_{n\in\N}\subset \bar{\Omega}_{\varphi}$ such that
    \[
    \lim_{n\to\infty}\delta_{\varphi}(y_n,y^*)=0 \ \text{for some} \ y^*\in\bar{\Omega}_{\varphi}.
    \]
    Then there exists $\{x_n\}_{n\in\N}\subset {\Omega}_{\varphi}$ such that $\pi(x_n)=y_n$. Pick $x^*\in{\Omega}_{\varphi}$ with $\pi(x^*)=y^*$. Thus we have
    \[
    \delta_{\varphi}(\pi(\lim_{n\to\infty}x_n),\pi(x^*))=0
    \]
    by the continuity of $\delta_\varphi$ and $\pi$. Since $\delta_\varphi$ is a metric on $\bar{\Omega}_\varphi$, we obtain $\pi(\lim_{n\to\infty}x_n)=\pi(x^*)$ and the invertibility of $\pi$ provides $\lim_{n\to\infty}x_n=x^*$. Thus we have
    \[
    \lim_{n\to\infty} \pi^{-1}(y_n)=\lim_{n\to\infty} x_n=x^*=\pi^{-1}(y^*),
    \]
    which completes the proof.
\end{proof}
Now we consider the quotient Aubry set for the case that $\mathrm{m}_h$ is not totally disconnected, i.e., contains some interval.
Naively, from Proposition~\ref{prop:iso_Aubry_m}, one expects that the quotient Aubry set corresponds to the quotient space of $\mathrm{m}_h$.
 Recall that, for $a,b\in\mathrm{m}_h$, we write $a\sim_{\mathrm{conn},h} b$ if $a$ and $b$ belong to the same connected component (denoted by $C(a)$) of $\mathrm{m}_h\subset \R$, which implies that $\sim_{\mathrm{conn},h}$ is an equivalent relation in $\mathrm{m}_h$.
 Set $\bar{\mathrm{m}}_h:=\mathrm{m}_h/\sim_{\mathrm{conn},h}$ endowed with the quotient topology. Note that $\bar{\mathrm{m}}_h$ is totally disconnected.
	As for the quotient Aubry set, it seems good that we introduce a pseudo-metric on $\mathrm{m}_h$ and obtain the induced quotient metric space.
	\begin{proposition}\label{prop:pseudo-metric}
		The function $\hat{d}_{\R}:\mathrm{m}_h\times{\mathrm{m}_h}\to\R$ given by
	\[
 	\hat{d}_{\R}(a,b)=d_H(C(a),C(b))
 	\]
 	is a pseudo-metric on ${\mathrm{m}_h}$, where $d_H$ stands for the Hausdorff distance.
	\end{proposition} 
	\begin{proof}
		The only non-trivial property is the triangle inequality.
		From the triangle inequality for Hausdorff distance,
		for $a,b,c\in\mathrm{m}_h$, we obtain
		\begin{align*}
			\hat{d}_{\R}(a,b)&=d_H(C(a),C(b))\\
			&\le d_H(C(a),C(c))+d_H(C(c),C(b))\\
			&=\hat{d}_{\R}(a,c)+\hat{d}_{\R}(c,b),
		\end{align*}
		which completes the proof.
	\end{proof}
	Using $\hat{d}_\R$, for $a,b\in\mathrm{m}_h$, we see that $a\sim_{\mathrm{conn}, h} b$ holds if and only if $\hat{d}_\R(a,b)=0$.
	Note that any connected component of $\mathrm{m}_h$ is closed and hence compact.
	 Therefore, the pseudo-metric $\hat{d}_{\R}$ on ${\mathrm{m}}_h$
	induces the metric of the quotient space $\bar{\mathrm{m}}_h$.
 \begin{proposition}[Main Theorem ~\ref{thm:Aubry_small}(ii)]\label{prop:Lip_Aubry_to_m}
 	If $h\in\mathcal{H}$ satisfies \eqref{eq:add_condition} for all $a\in\mathrm{m}_h$,
    then the map $\bar{\xi}:\bar{\mathrm{m}}_h\to \bar{\Omega}_\varphi; [a]\mapsto \pi(a^\infty)$ is well-defined and it is a homeomorphism.
 \end{proposition}
 \begin{proof}
 	By Theorems~\ref{theorem:not_equivalent} and \ref{thm:equivalent_interval}, we see that $a\sim_{\mathrm{conn},h} b$ holds if and only if $a^\infty\sim_\varphi b^\infty$, which implies the map $\bar{\xi}$ is well-defined.
	It is trivial that $\bar{\xi}$ is surjective, and the injectivity of $\bar{\xi}$ follows from Theorem~\ref{theorem:not_equivalent}.
    
    Since $\delta_{\varphi}$ does not depend the choice of representatives, taking into account Theorem~\ref{theorem:H_finite}, we compute
	\begin{align*}
		\delta_{\varphi}&(\bar{\xi}([a]),\bar{\xi}([b]))=\delta_{\varphi}(\pi(a^\infty),\pi(b^\infty))\\
		&=
		\inf\{H_{\varphi}(a'^\infty,b'^\infty)+H_{\varphi}(b'^\infty,a'^\infty)\mid a'^\infty\sim_\varphi a^\infty,\  b'^\infty\sim_\varphi b^\infty\}\\
		&\le \inf\{2L_{\varphi} d_{\R}(a',b'))\mid a' \sim_{\mathrm{conn},h} a,\  b' \sim_{\mathrm{conn},h} b\}\\
		&\le 2L_{\varphi} \hat{d}_{\R}(a,b),
	\end{align*}
	which implies that $\bar{\xi}$ is continuous.
    
    Now we show the continuity of $\bar{\xi}^{-1}$.
	Consider $y_n=\pi(a_n^\infty)\in \bar{\Omega}_\varphi\ (n\in\N)$ with $a_n\in\mathrm{m}_h$ such that $\delta_\varphi(y_n,y_*)\to 0\ (n\to \infty)$ for some $y_*=\pi(a_*^\infty)\in \bar{\Omega}_\varphi$ with some $a_*\in\mathrm{m}_h$.
    {Note that $\{a_n\}_{n\in\N}$ may not be a converge sequence.}
	Consider any convergent subsequence $\{a_{n_i}\}$ of $\{a_n\}$ and denote its accumulation point by $\tilde{a}\in\mathrm{m}_h$.
	By the continuity of $\delta_\varphi$ and $\pi$, we have $\delta_\varphi (\pi(\tilde{a}^\infty),\pi(a_*^\infty))=0$.
	Since $\delta_\varphi$ is a metric of $\bar{\Omega}_\varphi$, it holds that $\pi(\tilde{a}^\infty)=\pi(a_*^\infty)$, that is, $\tilde{a}^\infty\sim_\varphi a_*^\infty$.
	From Theorem~\ref{theorem:not_equivalent}, we obtain $[\tilde{a}]=[a_*]$.
	Then we compute
	\begin{align*}
		\lim_{i\to\infty} \hat{d}_\R(\xi^{-1}(y_{n_i}), \xi^{-1}(y_*))
		&=\lim_{i\to\infty} \hat{d}_\R(\xi^{-1}(\pi(a_{n_i}^\infty)), \xi^{-1}(\pi(a_*^\infty)))\\
		&=\lim_{i\to\infty} \hat{d}_\R([a_{n_i}], [a_*])\\
		&=\hat{d}_\R([\lim_{i\to\infty} a_{n_i}], [a_*])=\hat{d}_\R([\tilde{a}],[a_*])=0.
	\end{align*}
	Therefore, the map $\bar{\xi}$ is a homeomorphism and hence
	$\bar{\Omega}_\varphi$ is homeomorphic to $\bar{\mathrm{m}}_h$.
 \end{proof}

At the end of the section, we give the proof of Main Theorem~\ref{thm:Aubry_small}(iii).
\begin{theorem}[Main Theorem~\ref{thm:Aubry_small}(iii)]\label{thm:isometry}
	Let $\rho:\R\to\R$ be a $C^2$-function such that $\rho''>0$ on $\R$.
	{Let $\varphi_\rho$ be a Lipschitz potential depending only on the first two coordinates on $X$ given by
    $\varphi_\rho(\underline{x})=h_\rho(x_0,x_1)$ with $h_\rho\in\mathcal{H}\setminus C^1([0,1]^2,\R)$ of the form
	\[
		h_{\rho}(x,y)=\rho(x-y)+\frac{1}{2}|x-y|.
	\]
    }
	 Then, the quotient Aubry set $(\bar{\Omega}_{\varphi_\rho},\delta_{\varphi_\rho})$ of $\varphi_\rho$ is isometric to the unit interval $([0,1],d_\R)$.
\end{theorem}
\begin{proof}
    We first show that $h_{\rho}$ satisfies $(H_3)$ and $(H_4)$.
    Let
    \[
    h_1(x,y)=\rho(x-y), \ \text{and} \ h_2(x,y)=\frac{1}{2} |x-y|.
    \]
    Since $D_2D_1h_1=-\rho''(x-y)<0$, $h_1$ satisfies the twist condition and hence $(H_3)$ and $(H_4)$ hold for $h_1$.
    It is easily seen that
    $h_2(\xi_1,\eta_1) + h_2(\xi_2,\eta_2)
    \le h_2(\xi_1,\eta_2) + h_2(\xi_2,\eta_1)$
    for $\xi_1<\xi_2$ and $\eta_1 < \eta_2$.
    Combining the condition $(H_3)$ for $h_1$, we deduce that $h_{\rho}=h_1+h_2$ also satisfies $(H_3)$.
    
    Now, we want to check $(H_4)$ for $h_{\rho}$.
    Let $\tilde{\rho}(z)=\rho(z)+\frac{1}{2} |z|$ for $z\in[0,1]$. By $\rho''>0$ on $\R$ and the convexity of $|z|$, we see that $\tilde{\rho}$ is strictly convex. 
    Fix two distinct points $x_{-1},x_1\in[0,1]$.
    From the strict convexity of $\tilde{\rho}$, we have
    \begin{align}\label{eqn:strictly_convex}
    \begin{split}
    \tilde{\rho}\left(\frac{x_1-x_{-1}}{2}\right)
    &=\tilde{\rho}\left(\frac{(x_1-x_0)+(x_{0}-x_{-1})}{2}\right)\\
    &\le\frac{\tilde{\rho}(x_{1}-x_0)}{2}+\frac{\tilde{\rho}(x_0-x_{-1})}{2}
    \end{split}
    \end{align}
    for each $x_0\in[0,1]$ and the equality holds if and only if $x_1-x_0=x_0-x_{-1}$, i.e., $x_0=\frac{x_1+x_{-1}}{2}$. This implies that the function $h_{\rho}(x_{-1},x_0)+h_{\rho}(x_0,x_1)$ of $x_0\in[0,1]$ takes its minimum if and only if  $x_0=\frac{x_1+x_{-1}}{2}$.
    Assume that both $(x_{-1},x_0,x_1)$ and $(x'_{-1},x_0,x'_1)$ with $(x_{-1},x_0,x_1)\neq (x'_{-1},x_0,x'_1)$ are minimal for $h_{\rho}$.
    Then, it holds that $x_0=\frac{x_1+x_{-1}}{2}=\frac{x'_1+x'_{-1}}{2}$ and thus $x_{-1}-x'_{-1}=x'_1-x_1$. Note that $x_{-1}-x'_{-1}(=x'_1-x_1)\neq 0$ since if not we have $x_{-1}=x'_{-1}$ and $x_{1}=x'_1$ but $(x_{-1},x_0,x_1)\neq (x_{-1},x_0,x_1)$.
    Therefore, we obtain
    $(x_{-1}-x'_{-1})(x_{1}-x'_{1})=-(x_{-1}-x'_{-1})^2<0$,
    which implies $(H_4)$ holds for $h_{\rho}$.
    
    Next, we turn to the quotient Aubry set of $\varphi_{\rho}$.
    Since $h_{\rho}(x,x)=\rho(0)$ for all $x\in[0,1]$, we obtain $\mathrm{m}_{h_{\rho}}=[0,1], h_{\rho}^*=\rho(0)$ and $\Omega_{\varphi_{\rho}}=\{a^\infty\mid a\in[0,1]\}$.
    Fix $a,b\in[0,1]$ with $a\neq b$.
    From a similar discussion to \eqref{eqn:strictly_convex}, for any finite sequence $\{x_i\}_{i=0}^{n}$ with $x_0=a$ and $x_n=b$, we see that
    the function $\sum_{i=0}^{n-1} (h_{\rho}(x_i,x_{i+1})-\rho(0))$ of $x_1,\ldots,x_{n-1}\in[0,1]$ takes its minimum if and only if  $x_i=a+i\frac{b-a}{n}$.
    Therefore, for arbitrary $\varepsilon>0$, taking sufficiently large $n=n(\varepsilon)\in\N$ so that $\frac{|b-a|}{n}+\frac{|b-a|}{2^n}<\varepsilon$ and letting $x^*_i=a+i\frac{b-a}{n}$ for $i\in\N_0$, we see that $\underline{x}:=x_0^\ast x_1^\ast\ldots x_{n-1}^\ast b^\infty$ belongs to $ B(a^\infty,b^\infty,{n},\varepsilon)$ and thus
    \begin{align*}
    &\inf_{\underline{z}\in B(a^\infty,b^\infty,n,\varepsilon)} \sum_{i=0}^{n-1} (h_{\rho}(z_i,z_{i+1})-\rho(0))\\
	&= \sum_{i=0}^{n-1} (h_{\rho}(x^*_i,x^*_{i+1})-\rho(0))\\
	&= \left( \sum_{i=0}^{n-1} (\rho\left(\frac{a-b}{n}\right)-\rho(0))\right)
		+\left(\sum_{i=0}^{n-1} \frac{|b-a|}{2n}\right)\\
	&=\left(\sum_{i=0}^{n-1} (\rho\left(\frac{a-b}{n}\right)-\rho(0))\right)+\frac{1}{2}|b-a|.
	\end{align*}
    Taking into account the discussion in the proof of Proposition \ref{prop:explicit_estimate},
	we obtain
	\[
	H_{\varphi_{\rho}}(a^\infty,b^\infty)={\rho'(0)(a-b)} +\frac{1}{2}|b-a|.
	\]
	Similarly, it holds that $H_{\varphi_{\rho}}(b^\infty,a^\infty)={\rho'(0)(b-a)} +\frac{1}{2}|b-a|$
	and consequently we have
	\begin{align}\label{eqn:isometric}
		\delta_{\varphi_{\rho}}(a^\infty,b^\infty)=
		H_{\varphi_{\rho}}(a^\infty,b^\infty)+H_{\varphi_{\rho}}(b^\infty,a^\infty)
		=d_\R(a,b).
	\end{align}
	This implies that 
    \[
    a^\infty \nsim b ^\infty \  \text{for any} \ a,b \in \mathrm{m}_{h_{\rho}}\  \text{with}\ a\neq b,
    \]
    which means that the quotient Aubry set $\bar{\Omega}_{\varphi_{\rho}}$ is represented as
    \[
    \bar{\Omega}_{\varphi_{\rho}}=\{a^\infty\mid a\in[0,1]\}.
    \]
	Moreover, from the identity \eqref{eqn:isometric}, we deduce that
	the map $\xi:[0,1]\to \bar{\Omega}_{h_{\rho}}; a\mapsto a^\infty$ is an isometry with respect to the metrics $d_\R$ on $\mathrm{m}_{h_\rho}=[0,1]$ and $\delta_{h_{\rho}}$ on $\bar{\Omega}_{h_{\rho}}$, which implies that $(\bar{\Omega}_{h_{\rho}},\delta_{h_{\rho}})$ is isometric to the unit interval $([0,1],d_\R)$.
	\end{proof}

\appendix
\section{Uniqueness of calibrated subactions}
\label{sec:appendix}
Although the uniqueness of calibrated subactions is out of our main interests for this paper,
it is worth to note the following immediate consequence from the remark after the proof of Theorem~\ref{thm:relation_H_u}.
{Lopes et. al. \cite{LMST} showed that calibrated forward subactions play an important role in large deviation principles in the zero temperature limit. 
Here we study a related notion of calibrated subaction and prove its generic uniqueness.}
We only consider the full shift with $[0,1]$ defined in Section~\ref{sec:intro} and a Lipschitz continuous function $\varphi$ on $X=[0,1]^{\N_0}$ with the metric $d_X$.
\begin{proposition}\label{prop:uniqueness_calibrated _subaction}
        Let $\varphi$ be a Lipschitz continuous function on $X$.
        If $\Omega_\varphi$ consists of a single periodic orbit, then the calibrated subaction of Lipschitz function $\varphi$ is unique up to adding a constant and it is given by $\underline{y}\in X\mapsto H_\varphi(\underline{x},\underline{y})\in\R$ for each $\underline{x}\in\Omega_\varphi$
    (this map does not depend on the choice of $\underline{x}\in\Omega_\varphi$).
\end{proposition}
\begin{proof}
    Since $\Omega_\varphi$ consists of a single periodic orbit, the quotient Aubry set $\bar{\Omega}_\varphi$ is a singleton and, from Theorem~\ref{Peierl's_barrier_inf} and \eqref{eqn:reformulation_calibrated_subaction},
    we deduce that the calibrated subaction of Lipschitz function $\varphi$ is unique up to adding a constant.
    Combining this fact with Theorem~\ref{property_Peierl},
    we have the second claim.
\end{proof}
We then consider the set of $C^r$-functions ($r\ge 2$) with the twist condition,
\[
\mathscr{H}^r=\{h\in C^r([0,1]^2;\R)\mid D_2 D_1 h<0 \},
\]
equipped with the $C^r$-norm. 
It is shown that for generic potential $h$ in $\mathscr{H}^r$
the Aubry set $\Omega_h$ consists of a single fixed point
(Main Theorem 4 in \cite{KMS25}).
Therefore, taking into account Proposition~\ref{prop:uniqueness_calibrated _subaction}, we have the following generic uniqueness of calibrated subactions.
\begin{theorem}
	Let $r\ge 2$ be an integer. For the full shift with $[0,1]$, we have the following:
    \begin{itemize}
    \item[(i)] There is a $C^r$ open dense subset $\mathscr{O}$ in $\mathscr{H}^r$ such that for each $h\in \mathscr{O}$
    the Mather set and the Aubry set of $h$ consist of a single fixed point $a^\infty$ and
	the map $\underline{y}\in X\mapsto H_h(a^\infty,\underline{y})\in\R$ is the unique calibrated subaction of $h$ up to adding a constant.
    \item[(ii)] For arbitrary $h\in\mathscr{H}^r$ there is a $C^r$ open dense subset $\mathscr{V}_h$ in $C^r([0,1];\R)$ such that for each $V\in \mathscr{V}_h$
	the Mather set and the Aubry set of $h+V$ consist of a single fixed point $a^\infty$ and
	the map $\underline{y}\in X\mapsto H_{h+V}(a^\infty,\underline{y})\in\R$ is the unique calibrated subaction of $h+V$ up to adding a constant.
    \end{itemize}
\end{theorem}

\vspace*{33pt}

\noindent
\textbf{Acknowledgement.}~ 
The first author was partially supported by JSPS KAKENHI Grant Number  23H01081 and 23K19009. 
The second author was partially supported by JSPS KAKENHI Grant Number 25K17288.
The third author was partially supported by JSPS KAKENHI Grant Number 21K13816.

\vspace{11pt }
\noindent
\textbf{Data Availability.}~
Data sharing not applicable to this article as no datasets were generated or analyzed during the current study.

\bibliographystyle{alpha}
\bibliography{EOVP.bib}

\begin{bibdiv}
\begin{biblist}

\bib{Ban88}{incollection}{
      author={Bangert, V.},
       title={Mather sets for twist maps and geodesics on tori},
        date={1988},
   booktitle={Dynamics reported, {V}ol.\ 1},
      series={Dynam. Report. Ser. Dynam. Systems Appl.},
      volume={1},
   publisher={Wiley, Chichester},
       pages={1\ndash 56},
}

\bib{BCLMS11}{article}{
      author={Baraviera, A.~T.},
      author={Cioletti, L.~M.},
      author={Lopes, A.~O.},
      author={Mohr, Joana},
      author={Souza, Rafael~Rigao},
       title={On the general one-dimensional {{\(XY\)}} model: positive and
  zero temperature, selection and non-selection},
    language={English},
        date={2011},
        ISSN={0129-055X},
     journal={Rev. Math. Phys.},
      volume={23},
      number={10},
       pages={1063\ndash 1113},
}

\bib{BLL13}{book}{
      author={Baraviera, Alexandre},
      author={Leplaideur, Renaud},
      author={Lopes, Artur},
       title={Ergodic optimization, zero temperature limits and the max-plus
  algebra. {Paper} from the 29th {Brazilian} mathematics colloquium --
  29{{\(^{\text o}\)}} {Col{\'o}quio} {Brasileiro} de {Matem{\'a}tica}, {Rio}
  de {Janeiro}, {Brazil}, {July} 22 -- {August} 2, 2013},
    language={English},
      series={Publ. Mat. IMPA},
   publisher={Rio de Janeiro: Instituto Nacional de Matem{\'a}tica Pura e
  Aplicada (IMPA)},
        date={2013},
        ISBN={978-85-244-0356-9},
         url={impa.br/wp-content/uploads/2017/04/29CBM_07.pdf},
}

\bib{Bousch2000}{article}{
      author={Bousch, Thierry},
       title={The fish has no bones},
    language={French},
        date={2000},
        ISSN={0246-0203},
     journal={Ann. Inst. Henri Poincar{\'e}, Probab. Stat.},
      volume={36},
      number={4},
       pages={489\ndash 508},
         url={https://eudml.org/doc/77669},
}

\bib{ContrerasLopesThieullen2001}{article}{
      author={Contreras, G.},
      author={Lopes, A.~O.},
      author={Thieullen, Ph.},
       title={Lyapunov minimizing measures for expanding maps of the circle},
    language={English},
        date={2001},
        ISSN={0143-3857},
     journal={Ergodic Theory Dyn. Syst.},
      volume={21},
      number={5},
       pages={1379\ndash 1409},
         url={hdl.handle.net/10183/27431},
}

\bib{Contreras2016}{article}{
      author={Contreras, Gonzalo},
       title={Ground states are generically a periodic orbit},
        date={2016},
        ISSN={0020-9910,1432-1297},
     journal={Invent. Math.},
      volume={205},
      number={2},
       pages={383\ndash 412},
         url={https://doi.org/10.1007/s00222-015-0638-0},
}

\bib{CP02}{article}{
      author={Contreras, Gonzalo},
      author={Paternain, Gabriel~P.},
       title={Connecting orbits between static classes for generic {L}agrangian
  systems},
        date={2002},
        ISSN={0040-9383},
     journal={Topology},
      volume={41},
      number={4},
       pages={645\ndash 666},
         url={https://doi.org/10.1016/S0040-9383(00)00042-2},
}

\bib{CPV24}{article}{
      author={Carvalho, Maria},
      author={Pessil, Gustavo},
      author={Varandas, Paulo},
       title={A convex analysis approach to the metric mean dimension: limits
  of scaled pressures and variational principles},
    language={English},
        date={2024},
        ISSN={0001-8708},
     journal={Adv. Math.},
      volume={436},
       pages={54},
        note={Id/No 109407},
}

\bib{FFR09}{article}{
      author={Fathi, Albert},
      author={Figalli, Alessio},
      author={Rifford, Ludovic},
       title={On the {H}ausdorff dimension of the {M}ather quotient},
        date={2009},
        ISSN={0010-3640,1097-0312},
     journal={Comm. Pure Appl. Math.},
      volume={62},
      number={4},
       pages={445\ndash 500},
         url={https://doi.org/10.1002/cpa.20250},
}

\bib{GS2024}{article}{
      author={Gao, Rui},
      author={Shen, Weixiao},
       title={Low complexity of optimizing measures over an expanding circle
  map},
        date={2024},
        ISSN={0010-3616,1432-0916},
     journal={Comm. Math. Phys.},
      volume={405},
      number={10},
       pages={Paper No. 246, 17},
         url={https://doi.org/10.1007/s00220-024-05134-z},
}

\bib{HuntOtt1996B}{article}{
      author={Hunt, Brian~R.},
      author={Ott, Edward},
       title={Optimal periodic orbits of chaotic systems},
        date={1996Mar},
     journal={Phys. Rev. Lett.},
      volume={76},
       pages={2254\ndash 2257},
         url={https://link.aps.org/doi/10.1103/PhysRevLett.76.2254},
}

\bib{HuntOtt1996A}{article}{
      author={Hunt, Brian~R.},
      author={Ott, Edward},
       title={Optimal periodic orbits of chaotic systems occur at low period},
        date={1996Jul},
     journal={Phys. Rev. E},
      volume={54},
       pages={328\ndash 337},
         url={https://link.aps.org/doi/10.1103/PhysRevE.54.328},
}

\bib{Jenkinson2006}{article}{
      author={Jenkinson, Oliver},
       title={Ergodic optimization},
        date={2006},
        ISSN={1078-0947,1553-5231},
     journal={Discrete Contin. Dyn. Syst.},
      volume={15},
      number={1},
       pages={197\ndash 224},
         url={https://doi-org.kyoto-u.idm.oclc.org/10.3934/dcds.2006.15.197},
      review={\MR{2191393}},
}

\bib{Jenkinson19}{article}{
      author={Jenkinson, Oliver},
       title={Ergodic optimization in dynamical systems},
        date={2019},
        ISSN={0143-3857,1469-4417},
     journal={Ergodic Theory Dynam. Systems},
      volume={39},
      number={10},
       pages={2593\ndash 2618},
         url={https://doi.org/10.1017/etds.2017.142},
}

\bib{KMS25}{misc}{
      author={Kajihara, Yuika},
      author={Motonaga, Shoya},
      author={Shinoda, Mao},
       title={The aubry set for the xy model and typicality of periodic
  optimization for 2-locally constant potentials},
        date={2025},
         url={https://arxiv.org/abs/2504.00499},
}

\bib{LM14}{article}{
      author={Lopes, A.~O.},
      author={Mengue, J.~K.},
       title={Selection of measure and a large deviation principle for the
  general one-dimensional {{\(XY\)}} model},
    language={English},
        date={2014},
        ISSN={1468-9367},
     journal={Dyn. Syst.},
      volume={29},
      number={1},
       pages={24\ndash 39},
}

\bib{LMST}{article}{
      author={Lopes, A.~O.},
      author={Mohr, J.},
      author={Souza, R.~R.},
      author={Thieullen, Ph.},
       title={Negative entropy, zero temperature and {Markov} chains on the
  interval},
    language={English},
        date={2009},
        ISSN={1678-7544},
     journal={Bull. Braz. Math. Soc. (N.S.)},
      volume={40},
      number={1},
       pages={1\ndash 52},
}

\bib{LopesThieullen2003}{incollection}{
      author={Lopes, Artur~O.},
      author={Thieullen, Philippe},
       title={Sub-actions for {A}nosov diffeomorphisms},
        date={2003},
       pages={xix, 135\ndash 146},
        note={Geometric methods in dynamics. II},
      review={\MR{2040005}},
}

\bib{Mather03}{article}{
      author={Mather, John~N.},
       title={Total disconnectedness of the quotient {A}ubry set in low
  dimensions},
        date={2003},
        ISSN={0010-3640,1097-0312},
     journal={Comm. Pure Appl. Math.},
      volume={56},
      number={8},
       pages={1178\ndash 1183},
         url={https://doi.org/10.1002/cpa.10091},
}

\bib{Mather04}{article}{
      author={Mather, John~N.},
       title={Examples of {A}ubry sets},
        date={2004},
        ISSN={0143-3857,1469-4417},
     journal={Ergodic Theory Dynam. Systems},
      volume={24},
      number={5},
       pages={1667\ndash 1723},
         url={https://doi.org/10.1017/S0143385704000446},
}

\bib{Mather1982}{article}{
      author={Mather, John~N.},
       title={Existence of quasi-periodic orbits for twist homeomorphisms of
  the annulus},
    language={English},
        date={1982},
        ISSN={0040-9383},
     journal={Topology},
      volume={21},
       pages={457\ndash 467},
}

\bib{Mather91}{article}{
      author={Mather, John~N.},
       title={Action minimizing invariant measures for positive definite
  {L}agrangian systems},
        date={1991},
        ISSN={0025-5874,1432-1823},
     journal={Math. Z.},
      volume={207},
      number={2},
       pages={169\ndash 207},
         url={https://doi.org/10.1007/BF02571383},
}

\bib{Motonaga2025}{article}{
      author={Motonaga, Shoya},
       title={A convex dual approach to optimizations in ergodic theory and
  thermodynamic formalism},
    language={English},
        date={2025},
        ISSN={0951-7715},
     journal={Nonlinearity},
      volume={38},
      number={11},
       pages={23},
        note={Id/No 115015},
}

\bib{Sorrentino08}{article}{
      author={Sorrentino, Alfonso},
       title={On the total disconnectedness of the quotient {A}ubry set},
        date={2008},
        ISSN={0143-3857,1469-4417},
     journal={Ergodic Theory Dynam. Systems},
      volume={28},
      number={1},
       pages={267\ndash 290},
         url={https://doi.org/10.1017/S0143385707000351},
}

\bib{Sorrentino2015}{book}{
      author={Sorrentino, Alfonso},
       title={Action-minimizing methods in {Hamiltonian} dynamics. {An}
  introduction to {Aubry}-{Mather} theory},
    language={English},
      series={Math. Notes (Princeton)},
   publisher={Princeton, NJ: Princeton University Press},
        date={2015},
      volume={50},
        ISBN={978-0-691-16450-2; 978-1-400-86661-8},
}

\bib{Tsu19}{article}{
      author={Tsukamoto, Masaki},
       title={Mean dimension of full shifts},
    language={English},
        date={2019},
        ISSN={0021-2172},
     journal={Isr. J. Math.},
      volume={230},
      number={1},
       pages={183\ndash 193},
}

\bib{Tsu20}{article}{
      author={Tsukamoto, Masaki},
       title={Double variational principle for mean dimension with potential},
    language={English},
        date={2020},
        ISSN={0001-8708},
     journal={Adv. Math.},
      volume={361},
       pages={53},
        note={Id/No 106935},
}

\bib{YuanHunt1999}{article}{
      author={Yuan, Guocheng},
      author={Hunt, Brian~R.},
       title={Optimal orbits of hyperbolic systems},
    language={English},
        date={1999},
        ISSN={0951-7715},
     journal={Nonlinearity},
      volume={12},
      number={4},
       pages={1207\ndash 1224},
}

\bib{Yu22}{article}{
      author={Yu, Guo~Wei},
       title={Chaotic dynamics of monotone twist maps},
        date={2022},
        ISSN={1439-8516,1439-7617},
     journal={Acta Math. Sin. (Engl. Ser.)},
      volume={38},
      number={1},
       pages={179\ndash 204},
         url={https://doi.org/10.1007/s10114-022-0453-7},
}

\end{biblist}
\end{bibdiv}

\end{document}